\documentclass[oneside]{amsart}
\usepackage{amsmath,amsfonts,amssymb}%
\usepackage[all]{xy}
\usepackage[usenames]{color}
\usepackage{setspace}
\usepackage{comment}
\usepackage{url}

\usepackage{amsmath, amsthm, amsfonts, amssymb, stmaryrd}%
\usepackage{mathrsfs}

\usepackage{enumerate}

\usepackage{xparse, expl3}

\usepackage[colorlinks=true, linkcolor=blue, citecolor=green, urlcolor=blue]{hyperref}
\def\EM#1{\ensuremath{#1}}

\def\xx{{\EM{\mbf{x}}}}
\def\yy{{\EM{\mbf{y}}}}
\def\zz{{\EM{\mbf{z}}}}

\newcommand\rest[1][]{\EM{\upharpoonright_{#1}}}

\newcommand{\defn}[1]{{\bf{#1}}}

\def\Set{\text{SET}}
\DeclareDocumentCommand{\Powerset}{}{\mathcal{P}}

\newcommand{\Erdos}{\text{Erd\H{o}s}}
\newcommand{\Konig}{\text{K\H{o}nig}}

\DeclareMathOperator{\ORD}{ORD}
\DeclareMathOperator{\dom}{dom}
\DeclareMathOperator{\range}{ran}

\DeclareMathOperator{\Rel}{Rel}
\DeclareMathOperator{\arity}{ar}

\DeclareDocumentCommand{\TuringJump}{d[]}
{
\IfNoValueTF{#1}
	{
	\EM{\mbf{0'}}
	}
	{
	\EM{\mbf{#1'}}
	}
}

\DeclareDocumentCommand{\TuringDoubleJump}{d[]}
{
\IfNoValueTF{#1}
	{
	\EM{\mbf{0''}}
	}
	{
	\EM{\mbf{#1''}}
	}
}

\DeclareDocumentCommand{\TuringTripleJump}{d[]}
{
\IfNoValueTF{#1}
	{
	\EM{\mbf{0'''}}
	}
	{
	\EM{\mbf{#1'''}}
	}
}
\DeclareDocumentCommand{\TuringFourthJump}{d[]}
{
\IfNoValueTF{#1}
	{
	\EM{\mbf{0^{(4)}}}
	}
	{
	\EM{\mbf{#1^{(4)}}}
	}
}

\DeclareDocumentCommand{\TuringFifthJump}{d[]}
{
\IfNoValueTF{#1}
	{
	\EM{\mbf{0^{(5)}}}
	}
	{
	\EM{\mbf{#1^{(5)}}}
	}
}

\DeclareMathOperator{\Sunop}{Sun}
\DeclareDocumentCommand{\SunflowerComp}{d[]}
{
\IfNoValueTF{#1}
	{
	\EM{\Sunop}
	}
	{
	\EM{\Sunop(#1)}
	}
}

\DeclareMathOperator{\Sunflowerop}{SunFl}
\DeclareDocumentCommand{\Sunflower}{d[]}
{
\IfNoValueTF{#1}
	{
	\EM{\Sunflowerop}
	}
	{
	\EM{\Sunflowerop(#1)}
	}
}

\DeclareMathOperator{\FinSetCompop}{FS}
\DeclareDocumentCommand{\FinSetComp}{d[]}
{
\IfNoValueTF{#1}
	{
	\EM{\FinSetCompop}
	}
	{
	\EM{\FinSetCompop(#1)}
	}
}

\DeclareDocumentCommand{\FinSetSeqComp}{d[]}
{
\IfNoValueTF{#1}
	{
	\EM{\FinSetCompop_{<\w}}
	}
	{
	\EM{\FinSetCompop_{<\w}(#1)}
	}
}

\DeclareMathOperator{\distmathop}{d}

\DeclareDocumentCommand{\dist}{d[]}
{
\IfNoValueTF{#1}
{
    \EM{\distmathop}
}
{
    \EM{\distmathop_{#1}}
}
}

\DeclareDocumentCommand{\dual}{d()}
{
\IfNoValueTF{#1}
	{
	\EM{\hat{\ }}
	}
	{
	\EM{\hat{#1}}
	}
}

\DeclareDocumentCommand{\SizeAS}{d()}
{
\IfNoValueTF{#1}
	{
	\EM{|\cdot|}
	}
	{
	\EM{|#1|}
	}
}

\DeclareDocumentCommand{\compcK}{d[]}
{
\IfNoValueTF{#1}
	{
	\EM{\mathbb{K}}
	}
	{
	\EM{\mathbb{K}[#1]}
	}
}

\DeclareDocumentCommand{\AgeK}{d[]}
{
\IfNoValueTF{#1}
	{
	\EM{\mathbf{K}}
	}
	{
	\EM{\mathbf{K}[#1]}
	}
}

\DeclareDocumentCommand{\Rel}{d[]}
{
\IfNoValueTF{#1}
	{
	\EM{\mathcal{R}}
	}
	{
	\EM{\mathcal{R}_{#1}}
	}
}

\DeclareDocumentCommand{\Func}{d[]}
{
\IfNoValueTF{#1}
	{
	\EM{\mathcal{F}}
	}
	{
	\EM{\mathcal{F}_{#1}}
	}
}

\DeclareDocumentCommand{\ar}{d[]}
{
\IfNoValueTF{#1}
	{
	\EM{\arity}
	}
	{
	\EM{\arity_{#1}}
	}
}

\DeclareDocumentCommand{\Fn}{d<> d[] d()}
{
\IfNoValueTF{#3}
	{
	\EM{\textrm{Fn}(#1, #2)}
	}
	{
	\EM{\textrm{Fn}(#1, #2, #3)}
	}
}

\DeclareDocumentCommand{\Ball}{d[] d()}
{
\IfNoValueTF{#1}
{
    \IfNoValueTF{#2}
    {
    \EM{\mathrm{B}}
    }
    {
    \EM{\mathrm{B}(#2)}
    }
}
{
    \IfNoValueTF{#2}
    {
    \EM{\mathrm{B}_{#1}}
    }
    {
    \EM{\mathrm{B}_{#1}(#2)}
    }
}
}

\DeclareDocumentCommand{\Hereditary}{d()}
{
\IfNoValueTF{#1}
{
    \mathfrak{H}
}
{
    \mathfrak{H}(#1)
}
}

\DeclareMathOperator{\Sym}{Sym}
\DeclareDocumentCommand{\Perm}{d()}
{
\EM{\Sym(#1)}
}

\newcommand{\Fraisse}{\textrm{Fra\"iss\'e}}

\DeclareDocumentCommand{\tdcl}{d[] d()}
{
\IfNoValueTF{#1}
{
    \EM{\textbf{term}(#2)}
}
{
    \EM{\textbf{term}_{#1}(#2)}
}
}

\DeclareDocumentCommand{\gdcl}{d[] d()}
{
\IfNoValueTF{#1}
{
    \EM{\textrm{gcl}(#2)}
}
{
    \EM{\textrm{gcl}_{#1}(#2)}
}
}

\DeclareDocumentCommand{\qftp}{d[] d()}
{
\IfNoValueTF{#1}
{
    \EM{\quantfreetp(#2)}
}
{
    \EM{\quantfreetp_{#1}(#2)}
}
}

\DeclareMathOperator{\quantfreetp}{qtp}

\DeclareDocumentCommand{\Closure}{d<> d()}
{
\IfNoValueTF{#1}
{
    \EM{\textrm{cl}(#2)}
}
{
    \EM{\textrm{cl}_{#1}(#2)}
}
}

\DeclareDocumentCommand{\ClosureMap}{d[] d()}
{
\EM{\textrm{clMap}(#1, #2)}
}

\DeclareDocumentCommand{\HP}{}
{
    \textrm{(HP)}%
}
\DeclareDocumentCommand{\CHP}{d()}
{
\IfNoValueTF{#1}
{
    \textrm{(CHP)}
}
{
    \textrm{(\EM{#1}-CHP)}
}
}

\DeclareDocumentCommand{\JEP}{}
{%
    \textrm{(JEP)}%
}
\DeclareDocumentCommand{\DAP}{d()}
{%
\IfNoValueTF{#1}
{%
    \textrm{(DAP)}%
}%
{%
    \textrm{(\EM{#1}-DAP)}%
}%
}

\DeclareDocumentCommand{\SAP}{}
{
    \textrm{(SAP)}
}

\makeatletter
\newcommand{\dotminus}{\mathbin{\text{\@dotminus}}}

\newcommand{\@dotminus}{%
  \ooalign{\hidewidth\raise1ex\hbox{.}\hidewidth\cr$\m@th-$\cr}%
}

\DeclareMathOperator{\Lang}{\mathscr{L}}

\DeclareDocumentCommand{\quadiff}{}
{
    \EM{\quad\text{ if and only if }\quad}
}

\def\aa{{\EM{\mathbf{a}}}}
\def\bb{{\EM{\mathbf{b}}}}
\def\cc{{\EM{\mathbf{c}}}}
\def\dd{{\EM{\mathbf{d}}}}

\def\vv{{\EM{\mathbf{v}}}}

\def\cA{{\EM{\mathcal{A}}}}
\def\cB{{\EM{\mathcal{B}}}}
\def\cC{{\EM{\mathcal{C}}}}
\def\cD{{\EM{\mathcal{D}}}}
\def\cE{{\EM{\mathcal{E}}}}

\def\cG{{\EM{\mathcal{G}}}}
\def\cH{{\EM{\mathcal{H}}}}
\def\cK{{\EM{\mathcal{K}}}}
\def\cL{{\EM{\mathcal{L}}}}
\def\cM{{\EM{\mathcal{M}}}}
\def\cN{{\EM{\mathcal{N}}}}

\def\cP{{\EM{\mathcal{P}}}}

\def\cX{{\EM{\mathcal{X}}}}
\def\cY{{\EM{\mathcal{Y}}}}
\def\cZ{{\EM{\mathcal{Z}}}}

\def\cBL{{\EM{\mathcal{BL}}}}

\def\w{\EM{\omega}}

\def\Rationals{{\EM{{\mbb{Q}}}}}

\def\^{\EM{{}^{\And}}}

\def\And{\EM{\wedge}}

\def\<{\EM{\langle}}
\def\>{\EM{\rangle}}

\def\nl{\newline}

\def\mbb#1{\EM{\mathbb{#1}}}
\def\mbf#1{\EM{\mathop{\pmb{#1}}}}

\def\ul#1{\underline{#1}}

\def\st{\,:\,}
\def\:{\colon}

\providecommand{\dotdiv}{
  \mathbin{
    \vphantom{+}
    \text{
      \mathsurround=0pt 
      \ooalign{
        \noalign{\kern-.35ex}
        \hidewidth$\smash{\cdot}$\hidewidth\cr 
        \noalign{\kern.35ex}
        $-$\cr 
      }%
    }%
  }%
}

\DeclareDocumentCommand{\RightJustify}{m}{\hspace*{\fill}\mbox{#1}\penalty-9999\relax}

\newcounter{margincounter}
\DeclareDocumentCommand{\displaycounter}{}
	{{\arabic{margincounter}}}
\DeclareDocumentCommand{\incdisplaycounter}{}
	{{\stepcounter{margincounter}\arabic{margincounter}}}

\DeclareDocumentCommand{\DeclareComment}{m m m o d()}{%
\expandafter\DeclareDocumentCommand\csname Hide#1\endcsname {}
	{%
	\expandafter\DeclareDocumentCommand\csname #1\endcsname {+m} {}
	
	\expandafter\DeclareDocumentCommand\csname f#1\endcsname {+m} {}

	\expandafter\DeclareDocumentEnvironment{e#1} {} {} {}
	}

\expandafter\DeclareDocumentCommand\csname Show#1\endcsname {}
	{
	\expandafter\DeclareDocumentCommand\csname #1\endcsname {+m}
		{%
		\textcolor{#2}
			{ 
			{\tiny \bf (#3)}
			\IfValueT{#5}
				{%
				#5
				}
			####1
			}
		}

	\expandafter\DeclareDocumentCommand\csname f#1\endcsname {+m}
		{%
		\IfValueTF{#4}
			{
			\textcolor{#2}
			{\text{$\,^{(\incdisplaycounter{#4})}$}}
			\marginpar{\tiny\textcolor{#2}{
				{\text{\tiny $(\displaycounter{#4})$}}
				\text{\IfValueT{#5}{#5}
				####1}}}
			}
			{
			\textcolor{#2}
			{$\,^{(\incdisplaycounter)}$}
			\marginpar{\tiny\textcolor{#2}{
				{\tiny $(\displaycounter)$}
				\text{\IfValueT{#5}{#5}
				####1}}}
			}
		}

	\expandafter\DeclareDocumentEnvironment{e#1} {}
		{
		\textcolor{#2}
		\bgroup
		\IfValueT{#5}
			{%
			#5
			}
		}
		{
		\egroup
		}
	}

\csname Show#1\endcsname

}

\definecolor{NAColor}{rgb}{1.0,0.0,0.0}
\definecolor{ProblemColor}{rgb}{0.7,0.1,0.7}
\definecolor{TBDColor}{rgb}{0.0,0.0,0.8}
\definecolor{MathColor}{rgb}{0.0,0.4,0.1}
\definecolor{NateColor}{rgb}{0.0,0.5,1.0}
\definecolor{MostafaColor}{rgb}{1.0,0.0,1.0}
\definecolor{RefColor}{rgb}{1.0,0.0,1.0}
\definecolor{LaterColor}{rgb}{1.0,0.0,1.0}
\definecolor{EditColor}{rgb}{1.0,0.0,0}

\DeclareComment{NA}{NAColor}{\EM{\star}}
\DeclareComment{TBD}{TBDColor}{\EM{!}}
\DeclareComment{PROBLEM}{ProblemColor}{\EM{!!}}(\bf)
\DeclareComment{MATH}{MathColor}{\EM{\#}}
\DeclareComment{NATE}{NateColor}{\EM{\square}}(N:)
\DeclareComment{MOSTAFA}{MostafaColor}{\EM{\square}}(M:)
\DeclareComment{REF}{RefColor}{\EM{(r)}}(Ref:)
\DeclareComment{LATER}{LaterColor}{\EM{(L)}}(Later:)

\DeclareComment{EDIT}{EditColor}{\EM{(NE)}}(NEED TO EDIT:\ )

\DeclareDocumentCommand{\DeclareCounter}{m}%
{%
 	\ifcsname c@#1\endcsname%
	\else%
		\newcounter{#1}%
	\fi%
}

\DeclareDocumentCommand{\MyQED}{}{\qed}

\DeclareDocumentEnvironment{proof}{d() D[]{\MyQED} d<>}
{
\noindent\IfNoValueTF{#1}
{\emph{Proof.\!\!}}
{\emph{Proof\ #1.\ }}
}
{
\IfValueTF{#3}
	{%
	\ExplSyntaxOff
	\RightJustify{\EM{#2_{#3}}}
	\vspace{1ex}
	}
	{
	\RightJustify{\EM{#2}}
	\vspace{1ex}
	}
}

\DeclareCounter{ProofLabelcOUntEr}
\DeclareDocumentCommand{\ProofLabel}{}{%
\addtocounter{ProofLabelcOUntEr}{1}
\label{cUrrEntProoflAbEl\arabic{ProofLabelcOUntEr}}
}

\DeclareDocumentCommand{\ProofRef}{D<>{1}}
{%
\ref{cUrrEntProoflAbEl\arabic{ProofcOUntEr#1}}
}
\DeclareDocumentCommand{\ProofCref}{D<>{1}}
{%
\cref{cUrrEntProoflAbEl\arabic{ProofcOUntEr#1}}
}

\DeclareDocumentEnvironment{Proof}{d() D[]{\MyQED} D<>{1}}
{
\DeclareCounter{ProofcOUntEr#3}%
\setcounter{ProofcOUntEr#3}{\value{ProofLabelcOUntEr}}%
\begin{proof}(#1)[#2]<\ProofRef<#3>>%
}
{
\end{proof}
}

\ExplSyntaxOn
\def\TheoremDepth{section}

\DeclareDocumentCommand{\DeclareTheorem}{m o m o}{%
\IfNoValueTF{#4}
	{%
	\IfNoValueTF{#2}
		{%
		\newtheorem{#1vArIAblE}{#3}
		}
		{%
		\newtheorem{#1vArIAblE}[#2vArIAblE]{#3}
		}
	}
	{%
	\newtheorem{#1vArIAblE}{#3}[#4]%
	}
\newtheorem*{#1vArIAblE*}{#3}

\DeclareDocumentEnvironment{#1}{o o}

	{
	\IfValueT{##2}%
		{
		\begin{spacing}{##2}
		}
	\IfValueTF{##1}
		{
		\begin{#1vArIAblE}[##1]
		}
		{
		\begin{#1vArIAblE}
		}
	\ProofLabel
	}
	{
	\IfValueT{##2}%
		{
		\end{spacing}{##2}
		}
	\end{#1vArIAblE}
	}

\DeclareDocumentEnvironment{#1*}{o o}

	{
	\IfValueT{##2}%
		{
		\begin{spacing}{##2}
		}
	\IfValueTF{##1}
		{
		\begin{#1vArIAblE*}[##1]
		}
		{
		\begin{#1vArIAblE*}
		}
	}
	{
	\IfValueT{##2}%
		{
		\end{spacing}{##2}
		}
	\end{#1vArIAblE*}
	}
}
\ExplSyntaxOff

\theoremstyle{plain}
\DeclareTheorem{theorem}{Theorem}[\TheoremDepth]
\DeclareTheorem{proposition}[theorem]{Proposition}
\DeclareTheorem{lemma}[theorem]{Lemma}
\DeclareTheorem{corollary}[theorem]{Corollary}
\DeclareTheorem{claim}[theorem]{Claim}

\theoremstyle{definition}
\DeclareTheorem{definition}[theorem]{Definition}
\DeclareTheorem{axiom}[theorem]{Axiom}
\DeclareTheorem{convention}[theorem]{Convention}

\theoremstyle{remark}
\DeclareTheorem{remark}[theorem]{Remark}
\DeclareTheorem{example}[theorem]{Example}
\DeclareTheorem{question}[theorem]{Question}
\DeclareTheorem{conjecture}[theorem]{Conjecture}

\begin{document}

\title{Structured Sunflowers}

\begin{abstract}
We call an infinite structure $\cM$ \emph{sunflowerable} if whenever $\cM'$ is isomorphic to $\cM$ with underlying set $M'$, consisting of finite sets of bounded size, there is an $M_0 \subseteq M'$ such that $M_0$ is a sunflower and $\cM'\!\!\rest[M_0]$ is isomorphic to $\cM$. We give sufficient conditions on $\cM$ to show that $\cM$ is sunflowerable. These conditions allow us to show that several well-known structures are sunflowerable and give a complete characterization of the countable linear orderings which are sunflowerable. 

We show that a sunflowerable structure must be indivisible. This allows us to show that any \Fraisse\ limit which has the 3-disjoint amalgamation property and a single unary type must be indivisible. 

In addition to studying sunflowerability of infinite structures we also consider an analogous property of an age which we call the \emph{sunflower property}. We show that any sunflowerable structure must have an age with the sunflower property. We also give concrete bounds in the case that the age has the hereditary property, the 3-disjoint amalgamation property and is indivisible.  
\end{abstract}

\author{Nathanael Ackerman}
\address{Harvard University,
Cambridge, MA 02138, USA}
\email{nate@aleph0.net}

\author{Mary Leah Karker}
\address{Providence College, RI 02918 USA}
\email{mkarker@providence.edu}

\author{Mostafa Mirabi}
\address{The Taft School, Watertown, CT 06795, USA}
\email{mmirabi@wesleyan.edu}





\maketitle


\section{Introduction}

In combinatorial set theory a \emph{sunflower}, or a \emph{$\Delta$-system}, is a collection of sets any two pairs of which have a common intersection. The existence of infinite sunflowers has important applications in mathematical logic, such as for forcing large generic structures (for more information see \cite{Golshani}, \cite{MR4468686}, and \cite{Cohen-Generic-with-Functions_AGM}). 

There are two important situations when it is known that infinite sunflowers exist. First there is the following result of \Erdos\ and Rado which tells us that whenever we have an infinite collection of finite sets, all of the same size, that infinite collection must contain an infinite sunflower.

\begin{theorem}[\Erdos\ and Rado]
\label{Constant size sunflower lemma}
Suppose $n \in \w$ and $Y$ an infinite set consisting of sets of size $n$. Then there is a sunflower $Y_0 \subseteq Y$ with $|Y_0| = |Y|$. 
\end{theorem}
\begin{proof}
See \cite[Theorem I.ii]{MR111692} or \cite{MR2220838} p. 107 and p. 421. 
\end{proof}

Second we have the following which is known as the $\Delta$-systems lemma, and which considers the case when all sets are finite, but not necessarily of the same size. 

\begin{theorem}[Shanin]
\label{Delta-system lemma}
Suppose $Y$ is an uncountable collection of finite sets. Then there is a sunflower $Y_0 \subseteq Y$ with $|Y_0| = |Y|$. 
\end{theorem}
\begin{proof}
See \cite[Theorem I.9.18]{MR1940513}
\end{proof}


In this paper, we will be interested in collections of sets where the collection itself is an $\Lang$-structure, for some language $\Lang$, e.g. is a graph or a linear ordering. We say such a $\Lang$-structure is a \emph{structured sunflower} when its underlying set is a sunflower. 
We say an infinite $\Lang$-structure $\cM$ is \emph{sunflowerable} if whenever there is an isomorphic copy of $\cM$ whose underlying set consists of finite sets of the same size, then the isomorphic copy contains a structured sunflower. In this set up the \Erdos\ and Rado theorem says that if $\Lang_{\emptyset}$ is the empty language then every infinite $\Lang_\emptyset$-structure is sunflowerable. 

In Section \ref{Structured Sunflower Section} we will study variants of the property of being sunflowerable and will give a natural condition on an $\Lang$-structure, which is satisfied by dense linear orderings and the Rado graph, which guarantees a structure is sunflowerable. 

In Section \ref{Linear Ordered Sunflower Section} we focus on linear orderings. For both $\kappa$-scattered linear orderings and countable linear orderings we give a complete characterization of when a structure is sunflowerable. 

While the existence of infinite sunflowers is important in mathematical logic and combinatorial set theory, the existence of large finite sunflowers has significance in many other fields, including the study of circuit lower bounds, matrix multiplication, pseudo-randomness, and cryptography. For an overview of the connections to computer science see \cite{MR4334977}. One of the most important results in the study of finite sunflowers is the sunflower lemma of \Erdos\ and Rado (\cite{MR111692}). Ignoring the explicit bounds which were obtained, this lemma can be framed as follows. 

\begin{lemma}
\label{Sunflower Lemma}
For every $n \in \w$ and every finite set $X$ there is a finite set $Y$ such that for every set $Y'$ satisfying
\begin{itemize}
\item $|Y'| = |Y|$, and 

\item $Y'$ consists of sets of size $n$,
\end{itemize}
there is a set $X' \subseteq Y'$ where
\begin{itemize}
\item $X'$ is a sunflower, 

\item $|X'| = |X|$
\end{itemize}
\end{lemma}

If this lemma holds when we replace ``finite set'' by ``finite structure in an age $\AgeK$'' (in an appropriate way) then we say $\AgeK$ has the sunflower property. In Section \ref{Sunflower Property Of An Age} we will consider the sunflower property of an age and show that if a countable structure is sunflowerable then its age has the sunflower property. 

In Section \ref{Indivisibility} we also show that if a structure is sunflowerable, or if an age has the sunflower property, then it must be indivisible. This allows us to deduce that any structure which has a single unary type and the 3-disjoint amalgamation property is indivisible. 

Finally in Section \ref{Conjectures} we end with some conjectures. 

\subsection{Notation}

We work in a background model of set theory which we call $\Set$. We let $\Powerset_{<\w}(\Set)$ be the collection of finite sets. For $n \in \w$, we let $\Powerset_{n}(\Set)$ be the collection of sets of size $n$. We also let $[n] = \{0, 1, \dots, n-1\}$.  If $f$ is a function we let $\dom(f)$ denote its domain and $\range(f)$ denote its range. If $\alpha$ is an ordinal then $\alpha^*$ is the order type given by $(\alpha, \ni)$, i.e. the reverse ordering of $(\alpha, \in)$.

We fix a language $\Lang$ all of whose functions and relations have finite arity. We will use script letters to represent $\Lang$-structures, i.e. $\cA, \cM_0$, etc. and the corresponding Roman letters to represent their underlying sets, i.e. $A, M_0$, etc. If $\cA$ is a structure and $B \subseteq A$ then we let $\cA \rest[B]$ be the substructure of $\cA$ with underlying set $B$. We will use boldface to signify sequences of variables or of elements. We will write $\aa \subseteq \cM$ when $\aa$ is a sequence of elements of $\cM$. We will write $|\aa|$ for the length of such a sequence.  We write $\qftp[\cM](\aa)$ for the \defn{(complete) quantifier-free type} of $\aa$ in $\cM$, i.e. the collection of literals which are true of the tuple $\aa$ in $\cM$. All quantifier-free types considered in this paper will be complete.

\begin{definition}
We say a structure $\cM$ is \defn{ultrahomogeneous} if whenever 
\begin{itemize}
\item $\aa, \bb \subseteq \cM$ are tuples of size $< |M|$, 

\item  $\qftp[\cM](\aa) = \qftp[\cM](\bb)$ (and in particular $|\aa| = |\bb|$)

\item $a \in \cM$,
\end{itemize}
then there is some $b \in \cM$ such that $\qftp[\cM](\aa, a) = \qftp[\cM](\bb, b)$. 
\end{definition}

\section{Structured Sunflower}
\label{Structured Sunflower Section}

In this section we introduce one of the main properties of study, that of being \emph{sunflowerable}.

\begin{definition}
A \defn{sunflower} is a subset $Y \subseteq \Powerset_{<\w}(\Set)$ such that 
\[
(\exists r)\, (\forall p, q \in Y)\, p \neq q \rightarrow p \cap q = r.
\]
We say a set $Z \subseteq \Powerset_{<\w}(\Set)$ \defn{contains} a sunflower if there is a $Y \subseteq Z$ which is a sunflower. 
\end{definition}

\begin{definition}
Suppose $\cY$ is an $\Lang$-structure where $Y \subseteq \Powerset_{<\w}(\Set)$. We say a substructure $\cZ \subseteq \cY$ is a \defn{structured sunflower} if 
\begin{itemize}
\item $\cZ \cong \cY$, 

\item $Z$ is a sunflower. 

\end{itemize}

We say $\cY$ \defn{contains a structured sunflower} if there is a $\cZ \subseteq \cY$ which is a structured sunflower. 
\end{definition}

\begin{definition}
Suppose $\cM$ is an infinite $\Lang$-structure and either $n \in \w$ or $n$ is $< \w$. We say $\cM$ is \defn{$n$-sunflowerable} if whenever  
\begin{itemize}

\item $\cY$ is an $\Lang$-structure with $\cY \cong \cM$,

\item $Y \subseteq \Powerset_{n}(\Set)$, 

\end{itemize}
then $\cY$ contains a structured sunflower. 

We say $\cM$ is \defn{sunflowerable} if it is $n$-sunflowerable for every $n \in\w$. We say $\cM$ is \defn{strongly sunflowerable} if it is $<\w$-sunflowerable. 
\end{definition}

The following shows that the notion of being strongly sunflowerable is only interesting for uncountable structures. 
\begin{lemma}
\label{Countable structures and not strongly sunflowarable}
If $\cM$ is countable, then $\cM$ is not strongly sunflowerable.
\end{lemma}
\begin{proof}
Note that $\w \subseteq \Powerset_{<\w}(\w)$ but $\w$ contains no sunflower of size $3$. Therefore no structure on $\w$ can contain an infinite substructure which is a sunflower. 
\end{proof}

The following shows that the notions of $n$-sunflowerable form a hierarchy. 

\begin{lemma}
\label{n+2-sunflowrable implies n+1-sunflowerable}
Suppose $n \in \w$ and $\cM$ is $n+2$-sunflowerable. Then $\cM$ is $n+1$-sunflowerable. 
\end{lemma}
\begin{proof}
Suppose $\cM' \cong \cM$ and $M' \subseteq \Powerset_{n+1}(\Set)$. Let $x$ be any set not in $\bigcup M'$. Let $M^* = \{a \cup \{x\} \st a \in M'\}$ and let $\alpha\:M' \to M^*$ be the map where $\alpha(a) = a \cup \{x\}$. Then $\alpha$ is a bijection. Let $\cM^*$ be the $\Lang$-structure such that $\alpha$ is an isomorphism from $\cM'$ to $\cM^*$. Because $\cM$ is $n+2$-sunflowerable and $M^* \subseteq \Powerset_{n+2}(\Set)$ there must be a $M^\circ \subseteq M^*$ such that $M^\circ$ is a sunflower and $\cM^* \rest[M^\circ] \cong \cM^*$. Let $M^+ = \alpha^{-1}(M^\circ)$. We then have $M^+$ is a sunflower and $\cM'\rest[M^+] \cong \cM$. Therefore, as $\cM'$ was arbitrary, $\cM$ is $n+1$-sunflowerable. 
\end{proof}

We now give some conditions on a structure which will guarantee that it is sunflowerable.

\begin{definition}
We say a structure $\cM$ is \defn{$\kappa$-indivisible} if whenever $(P_i)_{i \in \alpha}$ partitions $M$ with $\alpha < \kappa$, then for some $i \in \alpha$,  there is a $Q \subseteq P_i$ such that $\cM\rest[Q] \cong \cM$. 

We say $\cM$ is \defn{indivisible} if it is $|M|$-indivisible. 
\end{definition}

Being indivisible is a strong form of robustness. We first show that if a structure is 2-sunflowerable then it is indivisible. 
\begin{proposition}
\label{2-sunfloweable implies indivisible}
If $\cM$ is $2$-sunflowerable then $\cM$ is indivisible. 
\end{proposition}
\begin{proof}
Suppose $\cM$ is $2$-sunflowerable, $|M| = \kappa$ and $(P_i)_{i \in \alpha}$ is a partition of $M$ with $\alpha < \kappa$. Let $|P_i| = \gamma_i$ for $i \in \alpha$. For $i \in \alpha$ let $Q_i = \{\{(0, i), (i+1, j)\}\}_{j \in \gamma_i} \subseteq \Powerset_{2}(\kappa \times \kappa)$. Note $|Q_i| = |P_i|$ and $Q_i$ is a sunflower. Let $Q = \bigcup_{i \in \alpha} Q_i$.

Note that for $i < j \in \alpha$, $Q_i \cap Q_j = \emptyset$. Therefore if for $i \in \alpha$, $\beta_i\:P_i \to Q_i$ is a bijection and  $\beta = \bigcup_{i \in \alpha} \beta_i$ then $\beta$ is a bijection from the underlying set of $\cM$ onto $Q$. Let $\cN$ be the $\Lang$-structure such that $\beta\:\cM \to\cN$ is an isomorphism. 

Because $\cM$ is $2$-sunflowerable there must be a $Y \subseteq Q$ which is a sunflower with $\cN\rest[Y] \cong \cM$. Then $|Y| = |M| = \kappa$. But then there must be distinct $a, b \in Y$ with $a \cap b \neq \emptyset$ as otherwise $|Y \cap Q_i| = 1$ for all $i \in\alpha$ and $|Y| \leq \alpha < \kappa$. 

As $|a| = |b| = 2$ we must have $|a\cap b| =1$. Hence for some $i \in \alpha$ we have $a \cap b = \{(0, i)\}$. But, because $Y$ is a sunflower, this implies for all distinct $c, d \in Y$ that $c \cap d = \{(0, i)\}$ and $c, d \in Q_i$. Therefore $Y \subseteq Q_i$. But then $\cM \rest[{\beta^{-1}[Y]}] \subseteq P_i$ is isomorphic to $\cM$ and witnesses the fact that $\cM$ is indivisible. 
\end{proof}  

We now show that being strongly sunflowerable is equivalent to being sunflowerable and uncountable. 

\begin{proposition}
The following are equivalent
\begin{itemize}
\item[(a)] $\cM$ is sunflowerable and uncountable, 

\item[(b)] $\cM$ is strongly sunflowerable.
\end{itemize}
\end{proposition}
\begin{proof}
By Lemma \ref{Countable structures and not strongly sunflowarable} we have that (b) implies (a). 

Now suppose (a). As $\cM$ is sunflowerable, by Proposition \ref{2-sunfloweable implies indivisible} $\cM$ is indivisible. Suppose $\cM' \cong \cM$ and $M' \subseteq \Powerset_{<\w}(\Set)$. Let $c\:M' \to \w$ be such that $c(a) = |a|$. As $\cM'$ is indivisible and uncountable there must be a $M^* \subseteq M'$ such that $|c``[M^*]| = 1$ and $\cM'\rest[M^*] \cong\cM'$. But then, as $\cM^*$ is sunflowerable there must be a $M^\circ \subseteq M^*$ such that $\cM^*\rest[M^\circ] \cong \cM^*$ and $M^\circ$ is a sunflower. Therefore $\cM$ is strongly sunflowerable and (b) holds.
\end{proof}

Proposition \ref{2-sunfloweable implies indivisible} has consequences for the types of languages which admit $2$-sunflowerable structures. 

\begin{lemma}
\label{Indivisible implies no constant symbols}
Suppose $\cM$ is an infinite $\Lang$-structure which indivisible. Then $\Lang$ has no constants. In particular if $\cM$ is $2$-sunflowerable the $\Lang$ has no constants. 
\end{lemma}
\begin{proof}
Suppose, to get a contradiction, that $c$ is a constant in $\Lang$. Let $P_0 = \{c^{\cM}\}$ and $P_1 = M \setminus P_0$. Then $\{P_0, P_1\}$ is a partition of $M$. As $\cM$ is indivisible $i \in \{0, 1\}$ and an $\cN \subseteq \cM$ such that $\cN \cong \cM$ and $N$ is contained in $P_i$. 

But as $\cN \subseteq \cM$ we have $c^{\cN} = c^{\cM}$ and so, in particular, $c^{\cM} \in N$. Therefore $N \not \subseteq P_1$. But $|P_0| = 1$ and so $N \not \subseteq P_0$, getting us our contradiction.  

The fact that $\Lang$ doesn't have constant symbols if $\cM$ is sunflowerable then follows from Proposition \ref{2-sunfloweable implies indivisible}. 
\end{proof}

We now give a property which will imply $\kappa$-indivisibility on $\kappa$-saturated structures.

\begin{definition}
Suppose $\cM$ is an $\Lang$-structure, $\aa$ is a tuple contained in $\cM$, and $a$ is an element of $\cM$ not in $\aa$. We then let $B_{\aa, a} = \{b \in \cM \st \qftp[\cM](\aa, a) = \qftp[\cM](\aa, b)\}$. 

We say $\cM$ has \defn{universal duplication of $\kappa$-quantifier-free types} if for all tuples $\aa \subseteq \cM$ of length $<\kappa$ and all $a \in \cM \setminus \aa$, there is a $C \subseteq B_{\aa, a}$ where $\cM\rest[C] \cong \cM$. 

We say $\cM$ has \defn{strong universal duplication of $\kappa$-quantifier-free types} if for all sequences $\aa \in \cM$ of length $< \kappa$ and $a \in \cM\setminus \aa$, we have $\cM\rest[B_{\aa, a}] \cong \cM$. 

If $\kappa = |M|$, then we omit mention of it. 
\end{definition}

Universal duplication of $\kappa$-quantifier-free types can been seen as a homogeneity requirement on the model. It is also worth noting that a structure having universal duplication of $\w$-quantifier-free types implies that all functions must be choice functions. 

\begin{lemma}
\label{Universal duplication of w-quantifier-free types implies choice function}
Suppose $\cM$ is an infinite $\Lang$-structure which has universal duplication of $\w$-quantifier-free types and $f$ is a function symbol in $\Lang$ of arity $n$. Then 
\[
\cM \models (\forall x_0, \dots, x_{n-1})\, \bigvee_{i \in [n]} f(x_0, \dots, x_{n-1}) = x_i. 
\]
\end{lemma}
\begin{proof}
Suppose, to get a contradiction, there is a tuple $\aa = (a_0, \dots, a_{n-1}) \subseteq \cM$ such that $f(a_0, \dots, a_{n-1})  \in \cM \setminus \{a_0,\dots, a_{n-1}\}$. Then, as $\cM$ has universal duplication of $\w$-quantifier-free types, there is a $N \subseteq B_{\aa, f(\aa)}$ with $\cN \cong \cM$. But $B_{\aa, f(\aa)} = \{b \st f(\aa) = b\} = \{f(\aa)\}$. In particular,  $|B_{\aa, f(\aa)}| = 1$ getting us our contradiction.  
\end{proof}

The following gives a basic condition when universal duplication of $\kappa$-quantifier-free types is the same as strong universal duplication of $\kappa$-quantifier-free types. 

\begin{definition}
\label{Definition 3-amalgamation of quantifier-free types}
We say a structure $\cM$ has \defn{$3$-amalgamation of quantifier-free types} if whenever $p(\xx, \yy)$, $q(\yy, \zz), r(\zz, \xx)$ are quantifier-free types realized in $\cM$ such that 
\begin{itemize}
\item $p(\xx, \yy) \rest[\xx] = r(\zz, \xx) \rest[\xx]$,

\item $p(\xx, \yy) \rest[\yy] = q(\yy, \zz) \rest[\yy]$,

\item $q(\yy, \zz) \rest[\zz] = r(\zz, \xx) \rest[\zz]$,

\end{itemize}
there is a quantifier-free type $a(\xx, \yy, \zz)$ realized in $\cM$ such that 
\begin{itemize}
\item $a(\xx, \yy, \zz) \rest[\xx\yy] = p(\xx, \yy)$, 

\item $a(\xx, \yy, \zz) \rest[\yy\zz] = q(\yy, \zz)$, 

\item $a(\xx, \yy, \zz) \rest[\zz\xx] = p(\zz, \xx)$. 
\end{itemize}
\end{definition}

$3$-amalgamation of quantifier-free types is a strengthening of the amalgamation property and, in general, there are notions of $n$-amalgamation for all $n$.  These are closely related to the property \DAP(n)\ of an age defined in Definition \ref{Age properties}. 

For our purposes we will be focused on saturated structures which admit quantifier-elimination (and hence are ultrahomogeneous). In this situation $3$-amalgamation implies strong universal duplication of quantifier-free types. 

\begin{lemma}
\label{Saturated quantifier-elimination and 3-amalgamation implies strong universal duplication of qf-types}
Suppose $\cM$ is $|M|$-saturated, has quantifier-elimination, and has $3$-amalgamation of quantifier-free types. Then $\cM$ has strong universal duplication of quantifier-free types. 
\end{lemma}
\begin{proof}
Let $|M| = \kappa$ and let $(p_i(\xx_i))_{i \in \kappa}$ be an enumeration of quantifier-free types realized in $\cM$ over tuples of size $<\kappa$. For $i \in \kappa$ let $\eta_{i} = (\exists \xx_i)\, \bigwedge p_i(\xx_i)$. Let $I = \{(i, j) \in \kappa^2 \st p_i(\xx_i) \subseteq q_j(\xx_j) \And \xx_j = \xx_i y_j\}$. For $(i, j) \in I$ let $\gamma_{i, j} = (\forall \xx_i)\, [\bigwedge p_i(\xx_i) \rightarrow (\exists y_j)\, \bigwedge q_j(\xx_iy_j)]$. 

Note as $\cM$ is saturated with quantifier-elimination $\cM \models \bigwedge_{i \in \kappa} \eta_i$ and $\cM \models \bigwedge_{(i, j) \in I} \gamma_{i, j}$. Further note that a standard back-and-forth argument shows that if $|\cN| = |\cM|$ and $\bigwedge_{i \in \kappa} \eta_i$ and $\bigwedge_{(i, j) \in I} \gamma_{i, j}$ then $\cN \cong \cM$. 

Suppose to get a contradiction that $\cM$ does not have strong universal duplication of quantifier-free types. Pick some $\aa, a$ such that $\cM \rest[B_{\aa, a}]$ is not isomorphic to $\cM$. Let $r = \qftp[\cM](\aa, a)$. 

There must then be some $\bb \in B_{\aa, a}$ and quantifier-free type $q(\xx, y)$ such that $\cM \models (\exists y)\, q(\bb, y)$ but $\cM \not \models (\exists y)\, r(\aa, y) \And q(\bb, y)$. But if $s = \qftp[\cM](\aa,\bb)$ then this contradicts the $3$-amalgamation of $q, r, s$. 
\end{proof}

Strong universal duplication of quantifier-free types is important as, in saturated models (of regular cardinal size) with quantifier-elimination, it implies indivisibility.

\begin{proposition}
\label{Strong universal duplication of qf types implies indivisible}
Suppose 
\begin{itemize}
\item $|M|$ is a regular cardinal, 

\item $\cM$ is a $|M|$-saturated structure with quantifier-elimination, 

\item $\cM$ has strong universal duplication of quantifier-free types. 
\end{itemize}
Then $\cM$ is indivisible. 
\end{proposition}
\begin{proof}
Let $|M| = \kappa$. Suppose, to get a contradiction, that $\cP$ is a partition of $M$ with $|\cP| < \kappa$ such that no element of $\cP$ contains a copy of $\cM$ as a subset. Let $(P_i)_{i \in \zeta}$ be an enumeration of $\cP$ where $|\cP| = \zeta$.  Suppose $(a_i)_{i \in \kappa}$ is an enumeration of $\cM$. We now define by induction on $\lambda < \zeta$ a sequence of elements $(\bb_i)_{i \in \lambda}$, a sequence of quantifier-free types $(r_i(\zz_i, t_i))_{i \in \lambda}$, and a sequence of quantifier-free types $(u_i(\vv_i, s_i))_{i \in \lambda}$ such that 
\begin{itemize}
\item $|\bb_\lambda| < \kappa$, 

\item $\bb_\lambda \in P_\lambda$, 

\item for $i < \lambda$ and each element of $c$ in $\bb_\lambda$ we have $\cM \models r_i(\bb_i, c)$, 
                  
\item $\cM \models (\forall y \in P_\lambda) \neg \bigwedge_{i \leq \lambda} r_i(\bb_i, y)$,  

\item $\cM \models (\forall \vv_\lambda, y)\, u_\lambda(\<\bb_i\>_{i\in \lambda}, y) \rightarrow \bigwedge_{i \leq \lambda} r_i(\bb_i,y)$

\item $\cM \models (\exists y)\, u_\lambda(\<\bb_i\>_{i\in \lambda}, y)$.  

\end{itemize}
For all $\lambda < \zeta$ let $D_\lambda = \{y \in \cM \st u_\lambda(\<\bb_i\>_{i\in \lambda}, y)\}$. \nl\nl
\ul{Case $\lambda = 0$:} \nl
As $P_0$ does not contain an isomorphic copy of $\cM$ there must be a tuple $\bb_0 \in P_0$ with $|\bb_0| < \kappa$ and a quantifier-free type $r_0(\xx, y)$ such that $\cM \models (\exists y)\, \bigwedge r_0(\bb_0, y)$ but $\cM \models (\forall y \in P_0)\, \neg \bigwedge r_0(\bb_0, y)$. Let $u_0 = r_0$. Note $D_0 = \{y \in \cM \st \cM \models \bigwedge u_0(\bb_0, y)\} \neq \emptyset$.  \nl\nl 
\ul{Case $\lambda = \alpha+1$:} \nl
By our inductive assumption, there must be an element $y \in D_\alpha$ and hence a $y \in \cM$ such that $\cM \models \bigwedge_{i \leq \alpha} r_i(\bb_i, y)$. Note, by the inductive hypothesis, for all $i \leq \alpha$, $P_i \cap D_\alpha = \emptyset$.  We now break into two subcases. \nl\nl
\ul{Subcase 1:} $D_{\alpha} \cap P_{\alpha+1} = \emptyset$.\nl
In this case let $\bb_{\alpha+1}$ be the empty string and $r_{\alpha+1}(y)$ be any unary quantifier-free type realized in $\cM$. Further let $u_{\alpha+1} = u_\alpha$.  \nl\nl
\ul{Subcase 2:} $D_\alpha \cap P_{\alpha+1} \neq \emptyset$.\nl 
As $\cM$ has strong universal duplication of quantifier-free types, $\cM \rest[D_\alpha] \cong \cM$. But $P_{\alpha+1}$ does not contain a copy of $\cM$. Therefore there must be a tuple $\bb_{\alpha+1}$ of size $<\kappa$ in $P_{\alpha+1} \cap D_\alpha$ and a quantifier-free type $r_{\alpha+1}(\bb_{\alpha+1}, y)$ which is realized in $\cM\rest[D_\alpha]$, say by an element $a$, but not realized in $P_{\alpha+1}$. In particular, $\bb_{\alpha+1}$ and $r_{\alpha+1}$ satisfy the inductive hypothesis. We therefore let $u_{\alpha+1}$ be the quantifier-free type of $\<\bb_i\>_{i \in \alpha +1}\^a$. 
\nl\nl 
\ul{Case $\lambda = \w \cdot \delta$:}\nl 
Let $u^*(y) = \bigcup_{i \in \w \cdot \delta} u_i(\<\bb_j\>_{j \leq i}, y)$. Note $u^*(y)$ is finitely consistent and so it is consistent. Further, as $|\bb_i| < \kappa$, $\w \cdot \delta < \kappa$ and $\kappa$ is regular we have $|\bigcup_{i < \w \cdot \delta} \bb_i| < \kappa$. Therefore $\cM \models (\exists y)\, u^*(y)$ as $\cM$ is $\kappa$-saturated. But this implies that $\bigcap_{i < \w \cdot \delta}D_{i} \neq \emptyset$. We now break into two subcases. \nl\nl
\ul{Subcase 1:} $\bigcap_{i < \w \cdot \delta} D_i \cap P_{\w \cdot \delta} = \emptyset$.\nl
In this case let $\bb_{\w \cdot \delta}$ be the empty sequence and $r_{\w \cdot \delta}(y)$ be any unary type realized in $\cM$ and $u_{\w \cdot \delta}(\vv_{\w \cdot \delta}, y) = \bigcup{i \in \w \cdot \delta} u_i(\vv_i, y)$. \nl\nl
\ul{Subcase 2:} $(\bigcap_{i < \w \cdot \delta} D_i) \cap P_{\alpha+1} \neq \emptyset$.\nl 
Note $\bigcap_{i < \w \cdot \delta} D_i = \{y \in \cM \st \cM \models u^*(y)\}$. So, $\cM$ has strong universal duplication of quantifier-free types, $\cM \rest[\bigcap_{i < \w \cdot \delta} D_i] \cong \cM$. But $P_{\w \cdot \delta}$ does not contain a copy of $\cM$. Therefore there must be a tuple $\bb_{\w \cdot \delta}$ in $P_{\w \cdot \delta} \cap (\bigcap_{i < \w \cdot \delta} D_i)$ and a quantifier-free type $r_{\w \cdot \delta}(\bb_{\w \cdot \delta}, y)$ which is realized in $\cM\rest[\bigcap_{i < \w \cdot \delta} D_i]$, by some tuple $a$, but not realized in $P_{\w \cdot \delta}$. In particular $\bb_{\w \cdot \delta}$,  $r_{\w \cdot \delta}$ and $u_{\w \cdot \delta}$ satisfy the inductive hypothesis.\nl\nl
Finally, note that the set $\bigcup_{i \in \zeta} r_{i}(\bb_i, y)$ is finitely consistent and $|\bigcup_{i \in \zeta} \bb_i| < \kappa$. Therefore, as $\cM$ is $\kappa$-saturated, there is a $y \in \cM$ such that $\cM \models \bigwedge_{i \in \zeta} r_i(\bb_i, y)$. But then for all $i \in \zeta$, $y \not \in P_i$, contradicting the fact that $\cP$ is a partition of $\cM$. 
\end{proof}

We now give a condition which will ensure being sunflowerable. This allows us to show that many well-known structures, such as the rational linear ordering and the Rado graph are sunflowerable.

\begin{theorem}
\label{Theorem showing when structure is sunflowerable}
Suppose 
\begin{itemize}

\item $\cM$ is an $\Lang$-structure, 

\item $|M|$ is an infinite regular cardinal

\item $\cM$ is indivisible, ultrahomogeneous, and has universal duplication of quantifier-free types. 
\end{itemize}
Then $\cM$ is sunflowerable. 
\end{theorem}
\begin{proof}
First note by Lemma \ref{Indivisible implies no constant symbols} and Lemma \ref{Universal duplication of w-quantifier-free types implies choice function} if $M_0 \subseteq M$ then $\cM\rest[M_0]$ is a substructure of $\cM$, i.e. $\cM$ restricted to a subset always gives rise to a substructure.

Suppose $n \in \w$, $Y \subseteq \Powerset_{n}(\Set)$ and $\cY$ is an $\Lang$-structure with underlying set $Y$ and $\cY \cong \cM$. As $|Y| = \kappa$ we can assume without loss of generality that $\bigcup Y \subseteq \kappa$ as well. 

We will now prove, by induction on $n$, that there is $Z \subseteq Y$ such that $\cY\rest[Z] \cong \cM$ and $Z$ is a sunflower.  \nl\nl
\ul{Case $n= 1$:} \nl
In this case $y_0 \cap y_1 = \emptyset$ for all distinct $y_0, y_1 \in Y$ and hence $Y$ is the desired sunflower. \nl\nl
\ul{Case $n = k+1$:}\nl
Suppose to get a contradiction that there is no $Z \subseteq Y$ which is a sunflower with $\cY\rest[Z] \cong \cM$. 

For each $x \in \kappa$ let $Y_x^\circ = \{y \in Y \st x \in y\}$. Let $Y_x = \{y\setminus \{x\} \st y \in Y_x^\circ\}$. Note $Y_x \subseteq \Powerset_{k}(\Set)$ and the map $u_x\:Y_x \to Y_x^{\circ}$ with $u_x(y) = y\cup \{x\}$ is a bijection. We let $\cY_x$ be the $\Lang$-structure on $Y_x$ such that for all atomic formulas $\psi$ of arity $\ell$ and all $z_0, \dots, z_{\ell-1} \in \cY_x$,
\[
\cY_x \models \psi(z_0, \dots, z_{\ell-1}) \quadiff \cY \models \psi(u_x(z_9), \dots, u_x(z_{\ell-1})). 
\]

If there is an $x \in \kappa$ and $C_x \subseteq Y_x$ is such that $\cY_x \rest[C_x] \cong \cM$, then by the inductive hypothesis applied to $\cY_x \rest[C_x]$ there is a $Z_x\subseteq C_x$ such that $\cY_x \rest[Z_x] \cong \cM$ and $Z_x$ is a sunflower. But then $u_x``[Z_x] \subseteq Y$ is a sunflower and $\cY\rest[{u_x``[Z]}] \cong \cM$ contradicting our assumption. 

It therefore suffices to consider the case when, for all $x \in X$, $Y_x$ contains no substructure isomorphic to $\cM$. 

Let $\aa = (a_i)_{i \in \kappa}$ be an enumeration of $\cM$. We define by induction the sequence $(z_i)_{i \in \kappa} \subseteq Y$ where 
\begin{itemize}
\item for all $\alpha \in \kappa$, $\qftp[\cM](\<z_i\>_{i \in \alpha}) = \qftp[\cM](\<a_i\>_{i \in \alpha})$, and

\item for all $i < j \in \kappa$, $z_i \cap z_j = \emptyset$.
\end{itemize}
\ul{Stage $0$:}\nl
Let $z_0 = a_0$. \nl\nl
\ul{Stage $\alpha$:}\nl
Let $r_\alpha = \sup \bigcup_{j < \alpha} z_j$. Note as $|z_i| = n$ for all $i \in \alpha$, and $|\alpha| < \kappa$ so $n \cdot \alpha < \kappa$ and hence, as $\kappa$ is an infinite regular cardinal, $r_\alpha < \kappa$.  Let $C^\circ_{\alpha} = \{b \in \cY \st \qftp[\cY](\<z_i\>_{i \in \alpha}, b) = \qftp[\cM'](\<a_i\>_{i \in \alpha}, a_{\alpha})\}$. Note $C^\circ_\alpha \neq \emptyset$ as $\cY$ is ultrahomogenerous. 

Let $C_{\alpha} \subseteq C_{\alpha}^\circ$ be such that $\cY\rest[C_{\alpha}] \cong \cY$. Note we can always find such a $C_{\alpha}$ as $\cY$ has universal duplication of quantifier-free types. 

Now suppose $C_{\alpha} \subseteq \bigcup_{i \leq r_\alpha} Y_{i}^\circ$. Then, as $\cY$ is indivisible, there must be an $j \leq r_{\alpha}$ and a $D_{\alpha} \subseteq C_{\alpha} \cap Y^\circ_{j}$ where $D_{\alpha} \cong \cY$. But then $u_j^{-1}(D_{\alpha}) \subseteq \cY_j$ and $u_j^{-1}(D_{\alpha}) \cong \cY$. Therefore by the inductive hypothesis applied to $u_j^{-1}(D_{\alpha})$ there is a $E_{\alpha} \subseteq u_j^{-1}(D_{\alpha})$ which is a sunflower and where $\cY_j \rest[E_{\alpha}] \cong \cY$. Therefore $u_j``[E_{\alpha}] \subseteq Y$ is a sunflower and $\cY\rest[u_j``[{E_{\alpha}]}] \cong \cY$. 

But we assumed no such subset existed. Therefore there must be some $z_{\alpha} \in C_{\alpha} \setminus \bigcup_{i \leq r_\alpha} Y_{i}^\circ$. Hence for all $i \leq r_\alpha$ we have $i \not \in z_{\alpha}$  as $z_\alpha \not \in Y_i$. Therefore, for all $j < \alpha$, $z_j \cap z_{\alpha} = \emptyset$.

Let $Z = \{z_i\}_{i \in \kappa}$. By construction the map $z_i \mapsto a_i$ is an isomorphism from $\cY\rest[Z] \to \cM$. Further as $z_i \cap z_j = \emptyset$ for all $i < j < \kappa$ we have $Z$ is a sunflower.  But this contradicts our assumption that no such $Z$ exist. 
\end{proof}

We also have the following corollaries. 
\begin{corollary}
\label{Corollary showing when structure is sunflowerable from strong universal duplication of qf types}
Suppose 
\begin{itemize}

\item $\cM$ is an $\Lang$-structure with $|M|$ a regular cardinal, 

\item $\cM$ is an $|M|$-saturated structure with quantifier-elimination, and

\item $\cM$ has strong universal duplication of quantifier-free types. 
\end{itemize}
Then $\cM$ is sunflowerable. 
\end{corollary}
\begin{proof}
This follows immediately from Proposition \ref{Strong universal duplication of qf types implies indivisible},  Theorem \ref{Theorem showing when structure is sunflowerable} and the fact that saturated structures with quantifier-elimination are ultrahomogeneous. 
\end{proof}

\begin{corollary}
\label{Corollary showing when structure is sunflowerable from 3-DAP}
Suppose 
\begin{itemize}

\item $\cM$ is an $\Lang$-structure with $|M|$ regular, 

\item $\cM$ is an $|M|$-saturated structure with quantifier-elimination, 

\item $\cM$ has $\DAP(3)$.   
\end{itemize}
Then $\cM$ is sunflowerable. 
\end{corollary}
\begin{proof}
This follows immediately from Lemma \ref{Saturated quantifier-elimination and 3-amalgamation implies strong universal duplication of qf-types} and  Corollary \ref{Corollary showing when structure is sunflowerable from strong universal duplication of qf types}. 
\end{proof}

We can apply Corollary \ref{Corollary showing when structure is sunflowerable from strong universal duplication of qf types} to give many examples of sunflowerable structures. 
\begin{lemma}
\label{kappa-dense linear orderings are sunflowerable}
If $\kappa$ is a regular cardinal and $(X, \leq)$ is a $\kappa$-saturated model of the theory of dense linear orderings without endpoints with $|X| = \kappa$ then $(X, \leq)$ is sunflowerable. 
\end{lemma} 
\begin{proof}
As the theory of dense linear orderings without endpoints has quantifier-elimination, by Corollary \ref{Corollary showing when structure is sunflowerable from strong universal duplication of qf types}, it suffices to show that $(X, \leq)$ has strong universal duplication of $\kappa$-quantifier-free types. 

Suppose $\aa$ is a sequence of elements of $X$ of size $<\kappa$ and suppose $a \in X \setminus \aa$. Let $q(x)$ be the type of $a$ over $\aa$ in $X$. Let $B_{\aa, a} = \{b \st (X, \leq) \models q(b)\}$.  We can then divide the elements of $\aa$ into two disjoint sets $A_<$, $A_>$ where $b_< < a < b_>$ for all $b_< \in A_<$ and $b_> \in A_>$. 

Let $p(x, y, z) = \{b_< < x \st b_< \in A_<\} \cup \{x < y, y < z\} \cup \{z < b_{>}\st b_> \in A_>\}$. Note for any $c, d \in B_{\aa, a}$ with $c < d$ the type $p(c, y, d)$ is finitely consistent and hence realized.  Therefore $(X, \leq)$ is a dense linear order and $|B_{\aa, a}| = |X| = \kappa$. 

Further, for all $c \in B_{\aa, a}$, the type $\{b_{<}<x \st b_< \in A_<\} \cup \{x <c\}$ and the type $\{b_{>}>x \st b_> \in A_>\} \cup \{x >c\}$ are consistent and hence realized. Therefore $(B_{\aa, a}, \leq)$ is a dense linear ordering without endpoints. 

Now suppose $C \subseteq B_{\aa, a}$ with $|C| < \kappa$ and $r(x)$ is a type over $C$. If $r(x)$ is of the form $x = c$ for some $c \in C$, then $r(x)$ is realized in $B_{\aa, a}$. If $r(x)$ is not of the form $x = c$  we can partition $C$ into two sets $C_{<}$ and $C_>$ such that $r(x)$ is equivalent to $\{c_< < x \st c_< \in C_<\} \cup \{x < c_> \st c_> \in C_>\}$. But then $q(x) \cup r(x)$ is finitely consistent and so there must be some $d$ such that $(X, \leq) \models q(d) \cup r(d)$. But such a $d$ must be in $B_{\aa, a}$ and so $(B_{\aa, a}, \leq)$ is $\kappa$-saturated. Hence $(B_{\aa, a}, \leq) \cong (X, \leq)$. So, as $\aa, a$ were arbitrary, $(X, \leq)$ has strong universal duplication of $\kappa$-quantifier-free types. 
\end{proof}

\begin{lemma}
\label{Saturated models of the Rado graph are sunflowerable}
If $\kappa$ is a regular cardinal and $(X, E)$ is a $\kappa$-saturated model of the theory of the Rado graph with $|X| = \kappa$, then $(X, E)$ is sunflowerable. 
\end{lemma}
\begin{proof}
This follows immediately from Corollary \ref{Corollary showing when structure is sunflowerable from 3-DAP} and the fact that the theory of the Rado graph has quantifier-elimination and the fact that any saturated model of the theory of the Rado graph has $\DAP(3)$. 
\end{proof}

\begin{example}
\label{Complicated hypergraph condition ensuring sunflowerablity}
For $1 \leq \ell \leq k$, let $K^k_\ell$ be the collection of all $k$-uniform hypergraphs such that all $\ell$-element subsets are in an edge. For $J \subseteq K^k_\ell$ let $\cH_{J}$ be the collection of all $k$-uniform hypergraphs which do not realize any element of $J$. Note $\cH_J$ is an age with \HP, \JEP, \SAP. Further, if $\ell \geq 3$ then $\cH_J$ has \DAP(3) as well. 

In particular, there is a \Fraisse\ limit $\cG_{J}$ of $\cH_J$ and by Corollary \ref{Corollary showing when structure is sunflowerable from 3-DAP} $\cG_J$ is sunflowerable. 
\end{example}

Note that these techniques don't apply to all ultrahomogeneous structures, even if the structures have free-amalagmation as the Rado graph does. 

For $n \geq 3$ let $\cH_n$ be the countable generic $K_n$-free graph.

\begin{lemma}
$\cH_n$ does not have universal duplication of quantifier-free types. 
\end{lemma}  
\begin{proof}
Let $a_0, \dots, a_{n-2}$ be a complete subgraph of $\cH_n$. If $\cH_n$ had universal duplication of quantifier-free types, then there would have to be a set $X \subseteq \{b \in \cH_n \st \qftp[\cH_n](a_0, \dots, a_{n-2}) = \qftp[\cH_n](a_0, \dots, a_{n-3}, b)\}$ such that $\cH_n \rest[X] \cong \cH_n$. In particular, there would have to be two $b_0, b_1$ such that $\qftp[\cH_n](a_0, \dots, a_{n-2}) = \qftp[\cH_n](a_0, \dots, a_{n-3}, b_0) = \qftp[\cH_n](a_0, \dots, a_{n-3}, b_1)$ and $\cH_n \models E(b_0, b_1)$. But then $\{a_i\}_{i \in [n-3]} \cup \{b_0\}$ is a complete $K_{n-1}$ and $\{a_i\}_{i \in [n-3]} \cup \{b_1\}$ is a complete $K_{n-1}$. So $\{a_i\}_{i \in [n-3]} \cup \{b_0, b_1\}$ is a complete $K_n$ contradicting the fact that $\cH_n$ is $K_n$-free. 
\end{proof}

\section{Linear Ordered Sunflowers}
\label{Linear Ordered Sunflower Section}

We now consider, for infinite cardinals $\kappa$, which linear orders of size $\kappa$ are sunflowerable.  

\begin{definition}
Suppose $\cY = (Y, \leq)$ is a linear ordering and for all $y \in Y$, $\cX_y = (X_y, \leq)$ is a linear ordering. We define $(\sum_{y \in \cY} \cX_y, \leq)$ to be the linear ordering where 
\begin{itemize}
\item the underlying set is $\{(y, x) \st y \in Y\text{ and } x \in X_y\}$, 

\item $(y_0, x_0) \leq (y_1, x_1)$ if either $y_0 < y_1$ or $y_0 = y_1$ and $x_0 \leq x_1$. 

\end{itemize}

We call $\sum_{y \in \cY} \cX_y$ a \defn{lexicographical sum}. 
\end{definition}

\begin{definition}
A linear ordering $(X, \leq)$ is \defn{$\kappa$-dense} if it has size at least $2$ and for all distinct $a, b \in X$, $|\{c \st a < c < b\}| = \kappa$. 
\end{definition}

Note that $(\Rationals, \leq)$ is $\w$-dense and the $\kappa$-dense linear orderings of size $\kappa$ are $\kappa$-sized analogs of $(\Rationals, \leq)$. This motivates the following definition. 

\begin{definition}
A linear ordering $(X, \leq)$ is \defn{$\kappa$-scattered} if there does not exist a $Y \subseteq X$ such that $(Y, \leq)$ is $\kappa$-dense. We say a linear ordering is \defn{scattered} if it is $\w$-scattered. 
\end{definition}

Hausdorff gave a characterization of scattered linear orderings which was generalized in \cite{MR2958935} to $\kappa$-scattered linear orderings.

\begin{definition}
Let $\cBL_\kappa$ be the collection of linear orderings which are either (a) of size $<\kappa$, (b) a well-ordering, or (c) an inverse of a well-ordering. Let $\cL_\kappa$ be the closure of $\cBL_{\kappa}$ under lexicographical sums ordered by elements of $\cBL_{\kappa}$. 
\end{definition}
\begin{proposition}[\cite{MR2958935} Theorem 3.10]
\label{Characterization of kappa-scattered linear orderings}
The following are equivalent for a linear ordering $L$. 
\begin{itemize}
\item $L$ is $\kappa$-scattered, 

\item $L \in \cL_\kappa$.
\end{itemize}            
\end{proposition}

\begin{definition}
Let $\cD_\kappa$ be the collection of linear orderings of size $< \kappa$. Consider the following collection of linear orderings defined by induction. $L_{0, \kappa}$ contains all isomorphic copies of elements of $\cD_\kappa$. If $L_{i, \kappa}$ has been defined for all $i \in \w \cdot \delta$ let $L_{\w \cdot \delta, \kappa} = \bigcup_{i \in \w \cdot \delta} L_{i,\kappa}$. 

If $L_{\alpha, \kappa}$ has been defined and $X$ is a linear ordering, let $L_{\alpha, \kappa}^X$ be the collection of linear orderings of the form $\sum_{x \in X} A_x$ where $A_x \in L_{\alpha, \kappa}$ for all $x \in X$. Let $L_{\alpha+1, \kappa} = \bigcup_{Y \in \cD_\kappa}L_{\alpha, \kappa}^{Y} \cup L_{\alpha, \kappa}^{\kappa} \cup L_{\alpha, \kappa}^{\kappa^*}$. Let $L_{\infty, \kappa} = \bigcup_{i \in \ORD} L_{i, \kappa}$.

For $A \in L_{\infty, \kappa}$, let $r_\kappa(A)$ be the least $\alpha$ such that $A \in L_{\alpha, \kappa}$. We call $r_\kappa(A)$ the \defn{$\kappa$-scattered construction rank} of $A$.
\end{definition} 

\begin{lemma}
\label{kappa-scattered description interms of kappa sequences}
$L_{\infty, \kappa}$ is the collection of $\kappa$-scattered linear orderings of size at most $\kappa$. 
\end{lemma}
\begin{proof}
It is immediate that every element of $L_{\infty, \kappa}$ has size at most $\kappa$. Further  $L_{\infty, \kappa} \subseteq \cL_\kappa$ and so by Proposition \ref{Characterization of kappa-scattered linear orderings} every element of $L_{\infty, \kappa}$ is $\kappa$-scattered. 

Also by Proposition \ref{Characterization of kappa-scattered linear orderings}, it suffices to show that for all ordinals $\beta$ of size at most $\kappa$ we have $\beta, \beta^* \in L_{\infty, \kappa}$. If $\beta \leq \kappa$ it is immediate that $\beta \in L_{\infty, \kappa}$. 

Now assume to get a contradiction that for some $\beta < \kappa^+$ we have $\beta \not \in L_{\infty, \kappa}$ and let $\beta$ be minimal as such. If $\beta = \alpha + 1$ then $\alpha \in L_{\infty, \kappa}$ and $\alpha + S_1 = \beta$ where $S_1$ is the $1$-point linear ordering (and hence in $\cD$). Therefore $\beta \in L_{\infty, \kappa}$, getting us a contradiction. 

But if $\beta$ is a limit, then there must be an increasing sequence of ordinals $(\zeta_i)_{i \in \eta}$ (for $\eta \leq \kappa$) such that $\beta = \sum_{i \in \eta} \zeta_i$. But as each $\zeta_i \in L_{\infty, \kappa}$ we have $\beta \in L_{\infty, \kappa}$. This gets us our contradiction, ensuring that for all $\beta < \kappa^+$, $\beta \in L_{\infty, \kappa}$. The case that $\beta^* \in L_{\infty, \kappa}$ is identical.  
\end{proof}

We now give an important property of $\kappa$-scattered linear orderings of size $\kappa$.

\begin{proposition}
\label{Main scattered lemma for kappa partitions}
Suppose $\cX$ is a $\kappa$-scattered linear ordering of size $\kappa$. Then one of three things happens. 
\begin{itemize}
\item[(a)] We can find a linear ordering $\cD$ with $|D| < \kappa$ and $\kappa$-scattered linear orderings $(\cX_i)_{i \in D}$ such that $\cX = \sum_{i \in \cD} \cX_i$ and for all $i \in D$, $\cX_i$ does not contain an isomorphic copy of $\cX$ as a subset, 

\item[(b)] We can find $\kappa$-scattered linear orderings $(\cX_i)_{i \in \kappa}$ such that $\cX = \sum_{i \in \kappa} \cX_i$ and for all $i \in \kappa$, $\cX_i$ does not contain a copy of $\cX$, 

\item[(c)] We can find $\kappa$-scattered linear orderings $(\cX_i)_{i \in \kappa}$ such that $\cX = \sum_{i \in \kappa^*} \cX_i$ and for all $i \in \kappa$, $\cX_i$ does not contain a copy of $\cX$. 
\end{itemize}
\end{proposition}
\begin{proof}
Let $\mbf{\cY}$ be the collection of $\kappa$-scattered linear orderings of size $\kappa$ which $\cX$ embeds into. Let $\cZ$ be an element of $\mbf{\cY}$ with $r_\kappa(\cZ)$ minimal. Let $\tau\:\cX \to \cZ$ be an embedding. Note $r_\kappa(\cZ) = \alpha +1$ for some ordinal $\alpha$ as no $\kappa$-sized linear ordering embeds into an element of $\cD_\kappa = L_{0, \kappa}$ and $\cZ$ has minimal $\kappa$-scattered construction rank among $\kappa$-scattered linear orderings containing a copy of $\cX$. We now have three cases depending on the form of $\cZ$. \nl\nl
\ul{Case 1:} $\cZ \in L_{\alpha, \kappa}^{Y}$ and $Y \in \cD_\kappa$\nl
There are then $(\cA_i)_{i \in Y} \subseteq L_{\alpha, \kappa}$ such that $\cZ = \sum_{i \in Y}\cA_i$.  Let  $\cX_i = \tau^{-1}(\cA_i)$. Note that if $\cX$ embedded into $\cX_i$ then $\cX$ would embed into $\cA_i$ contradicting the minimality of $r_\kappa(\cZ)$. Therefore $\cX_i$ satisfies condition (a) for all $i \in Y$. \nl\nl
\ul{Case 2:} $\cZ \in L_{\alpha, \kappa}^{\kappa}$\nl
There are then $\{\cA_i\}_{i \in \kappa} \subseteq L_{\alpha, \kappa}$ such that $\cZ = \sum_{i \in \kappa} \cA_i$. For $i \in \kappa$ let  $\cX_i = \tau^{-1}(\cA_i)$. Note that if $\cX$ embedded into $\cX_i$ then $\cX$ would embed into $\cA_i$ contradicting the minimality of $r(\cZ)$. Therefore $(\cX_i)_{i \in\kappa}$ satisfy condition (b). \nl\nl
\ul{Case 3:} $\cZ \in L_{\alpha, \kappa}^{\kappa^*}$\nl
There are then $\{\cA_i\}_{i \in \kappa} \subseteq L_{\alpha, \kappa}$ such that $\cZ = \sum_{i \in \kappa^*} \cA_i$. For $i \in \w$, let  $\cX_i = \tau^{-1}(\cA_i)$. Note that if $\cX$ embedded into $\cX_i$, then $\cX$ would embed into $\cA_i$ contradicting the minimality of $r(\cZ)$. Therefore $(\cX_i)_{i \in\kappa}$ satisfy condition (c). 
\end{proof}

We therefore have the following theorem. 
\begin{theorem}
\label{Characterization of kappa=scattered sunflowerable linear orderings}
Suppose $\cX = (X, \leq)$ is a $\kappa$-scattered linear ordering of size $\kappa$. Then the following are equivalent. 
\begin{itemize}
\item[(a)] $\cX$ is $2$-sunflowerable. 

\item[(b)] $\cX$ is sunflowerable.

\item[(c)] $\cX$ has order type $\kappa$ or $\kappa^*$. 

\end{itemize}
\end{theorem}
\begin{proof}
It is immediate that (b) implies (a). Suppose (c) holds. First note that any subset of $\kappa$ of size $\kappa$ has order type $\kappa$ and any subset of $\kappa^*$ of size $\kappa$ has order type $\kappa^*$. Therefore  both $\kappa$ and $\kappa^*$ are sunflowerable by Theorem \ref{Constant size sunflower lemma}. Hence (c) implies (b).

Now suppose (c) fails. By Proposition \ref{Main scattered lemma for kappa partitions}, there is a linear order $\cY$ such that either $|Y| < \kappa$ or the order type of $\cY$ is $\kappa$ or $\kappa^*$ and there is a collection $(\cX_i)_{i \in Y}$ such that $\cX = \sum_{i \in \cY} \cX_i$ but $\cX$ does not embed into any $\cX_i$. Let $(X_i, \leq) = \cX_i$. We can assume, without loss of generality,  that $X_i \cap X_j = \emptyset$ for all distinct $i, j \in Y$ and $X_i \cap Y = \emptyset$ for all $i \in Y$.  
   
Let $Z = Y \cup (\bigcup_{i \in Y} X_i)$. Note $|Z| = |X| = \kappa$. Let $A \subseteq \Powerset_2(Z)$ be the set $\{\{i, x\} \st i \in \kappa, x \in X_i\}$. Say that $\{i, x\} \leq_A \{j, y\}$ if and only if either $i  < j$ or $i = j$ and $x \leq y$. We then have $(A, \leq_A) \cong (X, \leq)$. 

Suppose, to get a contradiction, that $A_0 \subseteq A$ is a sunflower with $(A_0, \leq_A) \cong \cX$. Suppose $\{i, x\}, \{j, y\} \in A_0$ with $i, j \in Y$. If $\{i, x\} \cap \{j, y\} \neq \emptyset$ then $i = j$ and $A_0 \subseteq \{\{i, x\} \st x \in X_i\}$. But $\cX \cong (A_0, \leq_A)$ and so $\cX$ embeds into $\cX_i$ getting us our contradiction. 

Therefore we must have $\{i, x\} \cap \{j, y\} = \emptyset$. In particular, for all distinct $\{i, x\}, \{j, y\} \in A_0$ we have $i \neq j$. Therefore the order type of $(A_0, \leq_A)$ must be a suborder type of $Y$. But then the order type of $(A_0, \leq_A)$ is either size $< \kappa$ or of order type $\kappa$ or $\kappa^*$. This implies $\cX$ is either of size $< \kappa$ or of ordertype $\kappa$ or $\kappa^*$, contradicting our assumption on $\cX$. 

Therefore $\cX$ is not $2$-sunflowerable and (a) fails. Hence (a) implies (c).  
\end{proof} 

In the case when there is a saturated model $\Rationals_\kappa$ of size $\kappa$ of the theory of dense linear orderings we have a weakening of the notion of scattered. 

\begin{definition}
Let $\Rationals_\kappa$ be the unique (up to isomorphism) saturated model of the theory of dense linear orderings without endpoints of size $\kappa$. We say a linear ordering $(X, \leq)$ is \defn{$\Rationals_\kappa$-scattered} if there does not exist a subset $Y \subseteq X$ such that $(Y, \leq) \cong \Rationals_\kappa$. 
\end{definition}

The following is immediate by Lemma \ref{kappa-dense linear orderings are sunflowerable} and the fact that all linear orders of size $\kappa$ embed into $\Rationals_\kappa$. 

\begin{proposition}
If $\kappa$ is regular and $\cX$ is not $\Rationals_\kappa$-scattered then $\cX$ is sunflowerable. 
\end{proposition}

In the case when $\kappa = \w$ there is a unique $\w$-dense linear ordering (up to isomorphism), i.e. $(\Rationals, \leq)$. Further being $\Rationals$-scattered is the same as being scattered. Therefore we have the following characterization of sunflowerable countable linear orderings. 

\begin{theorem}
\label{Characterizing sunflowrable countable linear orderings}
Suppose $\cX$ is a countably infinite linear ordering. Then the following are equivalent. 
\begin{itemize}
\item $\cX$ is $2$-sunflowerable. 

\item $\cX$ is sunflowerable. 

\item One of the following is true of $\cX$
\begin{itemize}
\item $\cX$ is isomorphic to $(\w, \leq)$, 

\item $\cX$ is isomorphic to $(\w^*, \leq)$, 

\item $(\Rationals, \leq)$ embeds into $\cX$. 
\end{itemize}
\end{itemize}

\end{theorem} 
\begin{proof}
This is immediate from Lemma \ref{kappa-dense linear orderings are sunflowerable}, Theorem \ref{Characterization of kappa=scattered sunflowerable linear orderings}, the fact that any non-scattered linear ordering embeds $\Rationals$, and that all countable linear orderings embed into $\Rationals$. 
\end{proof} 

Note any $\kappa$-dense but non-sunflowerable linear ordering would have to become $\w$-dense in any larger model of set theory which collapses $\kappa$ to $\w$. Therefore every $\kappa$-dense linear ordering, whether or not it is sunflowerable, is sunflowerable in a forcing extension of the background model of set theory. 

We observe that being sunflowerable is not absolute. 

\begin{example}
Let $\cX = \w_1 \times \Rationals$ ordered lexographically. $\cX$ is $\w_1$-scattered and hence, by Theorem \ref{Characterization of kappa=scattered sunflowerable linear orderings}, not sunflowerable. But $\cX$ is a model of the theory of dense linear orderings without endpoints. Therefore in any larger model of set theory which collapses $\w_1$ to $\w$ we have $\cX$ is isomorphic to $\Rationals$ and hence is sunflowerable. 
\end{example}

\section{Sunflower Property Of An Age}
\label{Sunflower Property Of An Age}

It is well known that, via a compactness argument, the finite Ramsey theorem follows from the infinite Ramsey theorem. A similar argument shows that the sunflower lemma (without bounds) follows from the fact that all infinite sets of finite tuples of the same size contains an infinite sunflower. 

Further, this argument can be extended to sunflowerable structures and gives rise to the notion of an age having the sunflower property. In this section we will make this precise and also givconcrete upper bounds when the age has \DAP(3). 

\begin{remark}
For the rest of this section $\Lang$ will be a countable relational language. 
\end{remark}
\subsection{Properties of Ages}
\label{Properties of Ages}

We first recall various properties which will be important. 

\begin{definition}
An \defn{$\Lang$-age}, $\AgeK$, is a collection of finite structures closed under isomorphism. 

If $\cM$ is a structure then the \defn{age of $\cM$} is the collection of structures isomorphic to a finite substructure of $\cM$.
\end{definition}

\begin{definition}
\label{Age properties}
Suppose $\AgeK$ is an $\Lang$-age. 
\begin{itemize}
\item $\AgeK$ has the \defn{Hereditary Property}, \HP, if whenever $\cA \in \AgeK$ and $A_0 \subseteq A$ then $\cA\rest[A_0] \in \AgeK$. 

\item $\AgeK$ has the \defn{Joint Embedding Property}, \JEP, if whenever $\cA, \cB \in \AgeK$ there is a $\cC \in \AgeK$ and embeddings $i_{\cA}\:\cA \to \cC$ and $i_{\cB}\:\cB \to \cC$. 

\item $\AgeK$ has the \defn{$n$-disjoint amalgamation property} \DAP(n),\ if whenever 
\begin{itemize}
\item $\cA_0, \dots, \cA_{n-1} \in \AgeK$ 

\item for all $i, j \in [n]$ we have $\cA_i \rest[A_i \cap A_j] = \cA_j \rest[A_i \cap A_j]$
\end{itemize}
there is a $\cB \in \AgeK$ such that $\cB \rest[A_i] = \cA_i$ for all $i \in [n]$. 
\end{itemize}
\end{definition}

Note that the collection of $n$-disjoint amalgamation properties forms a hierarchy. In the literature there are two distinct ways to define this hierarchy of \DAP(n), see for example \cite{MR3911109} (which matches Definition \ref{Age properties}) and also the notion of \emph{free} in \cite{Aut(M)} and $n$-DAP in \cite{MR3835071} (which both match the notion of \emph{basic $n$-disjoint amalgamation} in \cite{MR3911109}).  However, the most common uses of these notions occur on ages which have $n$-DAP for all $n \in \w$. And, any such age will have $n$-DAP for all $n$ under one notion if and only if it has $n$-DAP for all $n$ in the other notion (see, for example, \cite[Proposition~3.6]{MR3911109}).  

Note a countable structure $\cM$ has $3$-amalgamation of quantifier-free types (Definition \ref{Definition 3-amalgamation of quantifier-free types}) if and only if its age has \DAP(3). 

Note the following is immediate. 
\begin{lemma}
\label{3-DAP implies amalgamation}
Suppose $\AgeK$ has \DAP(3), $\cA, \cB \in \AgeK$ and $\cA \rest[A \cap B] = \cB\rest[A\cap B]$ then there is a $\cC \in \AgeK$ such that $\cC \rest[A] = \cA$ and $\cC \rest[B] = \cB$. 
\end{lemma}
We call such a $\cC$ in Lemma \ref{3-DAP implies amalgamation} a \defn{disjoint amalgamation} of $\cA$ and $\cB$. 
The following is also immediate. 
\begin{lemma}
If $\AgeK$ has \DAP(3)\ then $\AgeK$ has \JEP. 
\end{lemma}

As we will see later, we will need the following property to hold in any age we care about. 
\begin{definition}
We say an age $\AgeK$ in $\Lang$ has \defn{unique unary quantifier-free types} if for all 
\begin{itemize}
\item $\cA, \cB \in \AgeK$, 

\item $a \in A$ and $b \in B$, 
\end{itemize}
we have $\qftp[\cA](a) = \qftp[\cB](b)$.
\end{definition}

The following lemma will be our main use of \DAP(3). 
\begin{lemma}
\label{Use of 3-DAP}
Suppose 
\begin{itemize}
\item $\AgeK$ has \DAP(3),

\item $\AgeK$ has unique unary types, 

\item $\cA, \cB \in \AgeK$ with $A \cap B = \emptyset$, 

\item $(a_i)_{i \in [n+1]}$ is an enumeration of $\cA$.
\end{itemize}
Then there is a $\cC \in \AgeK$ where 
\begin{itemize}
\item $\cA \subseteq \cC$ and $\cB \subseteq \cC$, and

\item for all $b \in \cB$, $\qftp[\cC]((a_i)_{i \in [n+1]}) = \qftp[\cC]((a_i)_{i \in [n]} \^ b)$.
\end{itemize}
\end{lemma}
\begin{proof}
Let $\{b_i\}_{i \in [\ell]}$ be an enumeration of $\cB$. We define a sequence of structures $\cC_i$ for $i \in [\ell] \cup \{-1\}$ as follows. Let $\cC_{-1} = \cA$. If $\cC_{k}$ is defined then we let $\cC_{k+1}\in \AgeK$ be a structure where 
\begin{itemize}
\item[(a)] $C_{k+1} = C_k \cup \{b_{k+1}\}$, 

\item[(b)] $\cC_{k+1} \rest[C_k] = \cC_k$, 

\item[(c)] $\cC_{k+1} \rest[{\{b_r\}_{r \in [k+2]}}] = \cB\rest[{\{b_r\}_{r \in [k+2]}}]$, 

\item[(d)] $\qftp[\cC_{k+1}]((a_i)_{i \in [n+1]}) = \qftp[\cC_{k+1}]((a_i)_{i \in [n]} \^ b_{k+1})$.

\end{itemize}
Note that the structures $\cC_k$ and $\cB\rest[{\{b_r\}_{r \in [k+2]}}]$ agree on their common intersection of $B\rest[{\{b_r\}_{r \in [k+1]}}]$. Similarly if $\cD$ is the structure with underlying set $\{a_i\}_{i \in [n]}\cup \{b_{k+1}\}$ such that $\cD \models \qftp[\cC_{k+1}]((a_i)_{i \in [n]} \^ b_{k+1})=  \qftp[\cC_{k+1}]((a_i)_{i \in [n+1]})$ then $\cD$ and $\cA$ agree on their common intersection which is $\{a_i\}_{i \in [n]}$. Finally note that $\cD$ and $\cB\rest[{\{b_r\}_{r \in [k+2]}}]$ agree on their common intersection of $\{b_{k+1}\}$ as $\AgeK$ has unique unary types. 

We can therefore apply \DAP(3)\ to find such a $\cC_{k+1} \in \AgeK$. Finally let $\cC = \cC_{\ell-1}$.  

\end{proof}

We now recall a couple important properties of ages.

\begin{definition}
We say an age $\AgeK$ is \defn{indivisible} at $\cA \in\AgeK$ if for all $n$ there is a $\cB \in \AgeK$ such that whenever $\alpha\:B \to [n+1]$ is a map, then there is an $\cA^+ \subseteq \cB$ where $\cA \cong \cA^+$ and $|\alpha[A^+]| = 1$. We call $\cB$ a \defn{witness to indivisibility} of $\cA$ in $\AgeK$ of order $n$. 
\end{definition}

When an age $\AgeK$ has unique unary quantifier-free types and \HP\ there is, up to isomorphism, a unique element $\cA_1$ of size $1$. In this situation being indivisible at all $\cB \in \AgeK$ is the same as a class having the $\cA_1$-Ramsey property (see \cite{MR1373655} Chapter 25.5 for a definition of the Ramsey property for an age).

The following lemma is also immediate. 
\begin{lemma}
\label{JEP and indivisible implies unqiue unary qf-types}
If $\AgeK$ has \JEP\ and is indivisible then $\AgeK$ has unique unary quantifier-free types.
\end{lemma}
\begin{proof}
Suppose $\AgeK$ has \JEP\ but does not have unique unary quantifier-free types. As $\AgeK$ has \JEP\ there must be an $\cA \in \AgeK$ and $a_0, a_1 \in \cA$ with $\qftp[\cA](a_0) \neq \qftp[\cA](a_1)$. 

Now suppose $\cB \in \AgeK$ is any structure containing an isomorphic copy of $\cA$. Let $\alpha\:B \to [2]$ be the map where $\alpha(b) = 0$ if and only if $\qftp[\cB](b) \neq \qftp[\cA](a_0)$. Then there can be no subset $B_0 \subseteq B$ with $|\alpha[B_0]| = 1$ and $\cB\rest[B_0] \cong \cA$. Hence $\AgeK$ is not indivisible. 
\end{proof}

The following definition will be important for determining bounds on the structured sunflower lemma. 

\begin{definition}
Let $\beta\:\w\times \w \to \w \cup \{\infty\}$. We say $\beta$ is a \defn{bound on indivisibility} if $\beta$ is non-decreasing in its first coordinate and if for all $\cA \in \AgeK$ and all $n\in \w$ either 
\begin{itemize}
\item $\AgeK$ is not indivisible at $\cA$ and $\beta(|A|, n) = \infty$, or 

\item there is a $\cB \in \AgeK$ with $|B| \leq \beta(|A|, n)$ and $\cB$ witnessing the indivisibility of $\cA$ in $\AgeK$ of order $n$. 

\end{itemize}

\end{definition}

We will want to use known results about indivisibility, along with Theorem \ref{Structued sunflowerable result}, to provide several examples of sunflowerable ages. However, as indivisibility is most often studied in infinite ultrahomogeneous structures, we present a result connecting the indivisibility of such a structure with the indivisibility of its age. 

\begin{lemma}
\label{Indivisible structure implies indivisible age}
Suppose $\cM$ is a countable structure which is indivisible and $\AgeK$ is the age of $\cM$. Then $\AgeK$ is indivisible.  
\end{lemma}
\begin{proof}
Suppose to get a contradiction that $\cM$ is indivisible but $\AgeK$ is not. There must then be an $n\in \w$ 
and an $\cA\in \AgeK$ such that for all $\cB \in \AgeK$ there is an $\alpha\:\cB \to [n+1]$ such that $\alpha^{-1}(i)$ does not contain a copy of $\cA$ for any $i \in [n+1]$. Call such a map $\alpha$ a \defn{bad map}. For each $\cB \in \AgeK$ let $Bad(\cB)$ be the collection of bad maps with range $[n+1]$ and domain $\cB$. Note for any $\cB$, $Bad(\cB)$ is finite. 

Let $(a_i)_{i \in \w}$ be an enumeration of $\cM$ without repetitions where $\cA \cong \{a_i\}_{i < |\cA|}$. For $k \in \w$ let $\cB_k = \cM \rest[\{a_i\}_{|\cA| + k}]$. Let $Bad = \bigcup_{k \in \w} Bad(\cB_k)$. Then $(Bad, \subseteq)$ is an infinite tree which is finitely branching. Therefore by \Konig's lemma there must be an infinite branch $(\alpha_k)_{i \in \w} \subseteq Bad$. If $\alpha = \bigcup_{k \in \w} \alpha_k$ we have $\alpha\:\cM \to [n+1]$ is a map such that $\alpha^{-1}(i)$ does not contain a copy of $\cA$ for any $i \in [n+1]$. But this contradicts the indivisibility of $\cM$. 
\end{proof}

\subsection{Sunflower Property}
\label{Age Structured Sunflower Section}

We now introduce the notion of the sunflower property. 
 
\begin{definition}
\label{structured n-sunflower property}
We say an age $\AgeK$ has the \defn{$n$-sunflower property} if for every $\cA \in \AgeK$ there is a $\cB \in \AgeK$ such that whenever $\cB' \cong \cB$ and $B' \subseteq \Powerset_{n}(\Set)$ then there is a $\cA' \subseteq \cB'$ such that $\cA' \cong \cA$ and $A'$ is a sunflower. We say $\cB$ \defn{witnesses} that $\AgeK$ has the $n$-sunflower property at $\cA$. 

If $\AgeK$ has the $n$-sunflower property for all $n$ then we say it has the \defn{sunflower property}. 
\end{definition}

Note because all elements of $\AgeK$ are finite, for each $\cA \in \AgeK$ with $A \subseteq \Powerset_n(\Set)$ we have $|\bigcup A| < \w$. Therefore there is an $\cA' \in \AgeK$ with $A' \subseteq \Powerset_n(\w)$ and an injective map $\alpha\: \bigcup A \to \w$ such that if $\alpha'\: A \to \Powerset_n(\w)$ is the map where for all $a\in A$,  $\alpha'(a) = \{\alpha(x) \st x \in a\}$ then $\alpha'$ is an isomorphism from $\cA$ to $\cA'$. Therefore, if we so choose, we could assume without loss of generality that all elements of $\AgeK$ in Definition \ref{structured n-sunflower property} have underlying sets contained in $\Powerset_n(\w)$.

\begin{example}
Let $\AgeK_{\text{LO}}$ be the class of linear orders. If $\cA, \cB\in \AgeK_{\text{LO}}$ and $\cA$ is of size $k$ then any subset of $\cB$ of size $k$ is isomorphic to $\cA$. Therefore, by the Theorem \ref{Constant size sunflower lemma}, $\AgeK_{\text{LO}}$ has the sunflower property. 
\end{example}

We now show the sunflower properties form a hierarchy.

\begin{lemma}
If $n \in \w$ and $\AgeK$ has the $n+2$-sunflower property, then $\AgeK$ has the $n+1$-sunflower property. 
\end{lemma}
\begin{proof}
Suppose $\cA \in \AgeK$ and $\cB$ witnesses that $\AgeK$ has the $n+2$-sunflower property at $\cA$. Suppose $\cB' \cong \cB$ and $B' \subseteq \Powerset_{n+1}(\Set)$. Let $X = \bigcup B'$ and $x \not \in X$. Let $\alpha\:\Powerset_{n+1}(X) \to \Powerset_{n+2}(X \cup \{x\})$ be the map where $\alpha(y) = y \cup \{x\}$. Then $\alpha$ is injective. Let $\cB^* = \alpha``[\cB']$. There must then be a subset $\cA^* \subseteq \cB^*$ such that $\cA^*$ is a sunflower and $\cA^* \cong \cA$. Let $\cA' = \alpha^{-1}(\cA^*)$. We then have $\cA' \cong \cA$. Also there is a $d$ such that for distinct $a, b\in \cA^*$ we have $d = a \cap b$. Further we have have $x \in d$. Therefore $\alpha^{-1}(a) \cap \alpha^{-1}(b) = d \setminus \{x\}$ for distinct $a, b$. Hence, as $\alpha\rest[\cA']$ is an isomorphism from $\cA'$ to $\cA^*$, we have the underlying set of $\cA'$ is a sunflower.  
\end{proof}

One of the main reasons why we restricted attention to ages which had unique unary quantifier-free types is that without this property the ages, in general, do not have even the $2$-sunflower property. 

\begin{lemma}
Suppose 
\begin{itemize}
\item $\AgeK$ does not have unique unary quantifier-free types, 

\item $\AgeK$ has \JEP, 

\item $\AgeK$ has at least one structure of size at least $3$. 
\end{itemize}

Then $\AgeK$ does not have the $2$-sunflower property.
\end{lemma}
\begin{proof}
Suppose $\AgeK$ does have the $2$-sunflower property to get a contradiction. By our assumptions on $\AgeK$ there must be an $\cA \in \AgeK$ with $|A| \geq 3$ along with $a, b \in A$ such that $\qftp(a) \neq \qftp(b)$. 

Then there is a $\cB$ witnessing that $\AgeK$ has the $2$-sunflower property at $\cA$. Pick $a \in \cA$ and let $E$ be the collection of elements $c \in B$ such that $\qftp(c) = \qftp(a)$. 

Let $B$ be the underlying set of $\cB$. Let $Y = \{(0, b) \st b \in E\} \cup \{(1, b) \st b \in B \setminus E\}$. Let $\cB'$ be the structure with underlying set $Y$ such that $\pi_1\:Y \to B$ given by projection onto the second coordinate, is an isomorphism. Then every sunflower contained in $Y$ either has size $\leq 2$, is contained in $E$ or is contained in $B \setminus E$. Therefore, as $|\cB| \geq |\cA| \geq 3$, no subset of $\cB'$ is isomorphic to $\cB$ and is a sunflower. 
\end{proof}

\subsection{Sunflower Property from Sunflowerable Structures}

We now show that the age of a countable sunflowerable structure has the sunflower property. This will allow us to give several examples of ages with the sunflower property.

\begin{theorem}                       
\label{Sunflowerable structure has an age with the sunflower property}
Suppose $\cK$ is an age and $n \in \w$. Further suppose $\cM$ is a countable $\Lang$-structure where 
\begin{itemize}
\item the age of $\cM$ is $\cK$, 

\item $\cM$ is $n$-sunflowerable. 

\end{itemize}
Then $\cK$ has the $n$-sunflower property. 

\end{theorem}
\begin{proof}
Note as there is a structure whose age is $\cK$, $\cK$ satisfies (HP) and (JEP). Suppose, to get a contradiction, that $\cK$ does not have the $n$-sunflower property. Then there is some $\cA \in \cK$ such that for every $\cB \in \cK$ there is a structure $\cB' \cong \cB$ with $B' \subseteq \Powerset_n(\Set)$ and whenever $\cA' \cong \cA$ with $\cA' \subseteq \cB'$ we have that $A'$ is not a sunflower. 

Let $(c_i)_{i \in \w}$ be an enumeration of $\cM$. For $\ell \in \w$ let $\cC_\ell$ be the substructure of $\cM$ where $C_\ell = \{c_i\}_{i \in [\ell]}$. Note as the age of $\cM$ is $\cK$ we have $\cC_\ell \in \cK$. 

If $X$ is a set, let $D_\ell^X$ be the collection of structures $\cC'_\ell$ such that $\cC'_\ell \cong \cC_\ell$, $C'_\ell \subseteq \Powerset_{n}(X)$ and no subset of $\cC'_\ell$ isomorphic to $\cA$ is a sunflower. Note for each $X \in \Set$ and each $\cE \in \cD_\ell^X$ we can find a subset $X_0 \subseteq [n \cdot \ell]$ and a bijection $\zeta \:\bigcup E \to X_0$ and an element $\cE_0 \in \cD_\ell^{[n \cdot \ell]}$ such that if $\zeta^* \:\Powerset(\bigcup E) \to \Powerset(X_0)$ is the corresponding bijection then $\zeta^*\:\cE \to \cE_0$ is an isomorphism of structures. In particular every element of $D^X_\ell$ is isomorphic to an element of $D_\ell^{[n \cdot \ell]}$ via an isomorphism which preserves the structure of the underlying set. Therefore for all $\ell \in \w$ if  $D_\ell = D_\ell^{[n \cdot \ell]}$ we have $D_\ell \neq \emptyset$.      

For $\ell \in \w$, we will define a tree $(T_\ell, \subseteq)$ where the $k$th level of the tree is a set $F_{\ell, k}$, which we will define by induction. Let $F_{\ell, 0} = D_\ell$. If $\cE \in F_{\ell, k}$ let $F^*(\cE)$ be the collection of $\cG \in D_{\ell + k+1}$ such that $\cE \subseteq \cG$. We let $F_{\ell, k+1} = \bigcup \{F^*(\cE) \st \cE \in F_{\ell, k}\}$. 

Note that if $\cE \in D_{\ell +k}$ then by permuting the elements $[(\ell + k) \cdot n]$ we can find an isomorphic copy of $\cE^* \cong \cE$ with $\cE^* \in F_{\ell, k}$. Therefore, as $D_{\ell+k} \neq \emptyset$ for all $k\in \w$, we have $F_{\ell, k} \neq \emptyset$ for all $k \in \w$. But then $T_{\ell}$ is infinite.    

For all $k \in \w$, as $D_{\ell + k}$ is finite and so $F_{\ell, k}$ is finite as well. Therefore $(T_\ell, \subseteq)$ is finitely branching. Hence, by \Konig's lemma, we must have an infinite branch $(\cE_i)_{i \in \w}$ in $T_\ell$.  Note that for $k \in \w$, $|E_{k+1} \setminus E_k| = 1$. Let $\cE = \bigcup_{i \in \w} \cE_i$. We then have $\cE \cong \cM$ where the unique element of $E_{k +1} \setminus E_k$ gets mapped to $c_{\ell+k+1}$.  


Because $\cM$ is $n$-sunflowerable, there must be an infinite sunflower $E^- \subseteq E$ such that $\iota\:\cE^- \cong \cM$. As the age of $\cM$ is $\cK$ there must be a subset $A^- \subseteq E^-$ with $\cA^- \cong \cA$. In particular, as $A^-$ is a subset of a sunflower, it is itself a sunflower. 

But there then must be some $k \in \w$ such that $A^- \subseteq E_k$. However $\cE_k \in F_{k, \ell} \subseteq D_{\ell + k}$ and so $\cE_k$ contains no substructure isomorphic to $\cA$ which is also a sunflower, getting us our contradiction. 
\end{proof}

We can use Theorem \ref{Sunflowerable structure has an age with the sunflower property} to show that several ages have the sunflower property. 

\begin{lemma}
The following ages have the sunflower property. 
\begin{itemize}
\item[(a)] The collection of all finite graphs

\item[(b)] The age $\cH_J$ from Example \ref{Complicated hypergraph condition ensuring sunflowerablity}

\end{itemize}

\end{lemma}
\begin{proof}
Item (a) follows from Lemma \ref{Saturated models of the Rado graph are sunflowerable} and Theorem \ref{Sunflowerable structure has an age with the sunflower property}, and (b) follows from Example \ref{Complicated hypergraph condition ensuring sunflowerablity} and Theorem \ref{Sunflowerable structure has an age with the sunflower property}. 
\end{proof}

\subsection{\DAP(3)}

Note that from Theorem \ref{Sunflowerable structure has an age with the sunflower property} we can deduce general properties which will ensure an age has the sunflower property. In particular we have the following. 

\begin{theorem}
\label{3-DAP implies sunflower property}
Suppose $\AgeK$ is an age with only countably many isomorphism classes, which has unique unary quantifier-free types, and which has \HP\ and \DAP(3). Then $\AgeK$ has the sunflower property. 
\end{theorem}
\begin{proof}
Because we are in a relational language, \DAP(3)\  implies \JEP. Therefore $\AgeK$ has a \Fraisse\ limit $\cG$. But then by Corollary \ref{Corollary showing when structure is sunflowerable from 3-DAP}, $\cG$ is sunflowerable and so by Theorem \ref{Sunflowerable structure has an age with the sunflower property}  $\AgeK$ has the sunflower property. 
\end{proof} 

In the case of ages with \HP\ and \DAP(3)\ we can also give concrete bounds on the size of a structure witnessing the sunflower property. 

\begin{theorem}
\label{Structued sunflowerable result}
Suppose 
\begin{itemize}

\item $\AgeK$ has \HP\ and \DAP(3), 

\item for all $\cA \in \AgeK$ with $|\cA| > 2$, $\AgeK$ is indivisible at $\cA$, 

\item $\beta$ is a bound on indivisibility for $\AgeK$, and 

\item for $\cA \in \AgeK$, $\gamma_\cA$ is the function where
\begin{itemize}
\item $\gamma_{\cA}(1) = |A|$, and 

\item $\gamma_{\cA}(n+1) =  \beta(\gamma_{\cA}(n), (n+1)\cdot |A|)^{\leq |A|}$.
\end{itemize}

\end{itemize}
Then for all $\cA \in\AgeK$ and $n \in \w$ there is a $\cB \in \AgeK$ with $|B| \leq \gamma_{\cA}(n)$ such that $\cB$ witnesses the $n$-sunflower property of $\AgeK$ at $\cA$.
\end{theorem}
\begin{proof}
We prove this by induction on $n$. First note that if $n = 1$ and $\cA \in \AgeK$ with $A \subseteq \Powerset_1(\Set)$ then $A$ is a sunflower. Therefore $\cA$ is a witness to this fact and $|A| \leq \gamma_{\cA}(1) = |A|$.


We now assume the result holds for $n$ in order to show it holds for $n+1$. Suppose $\cA \in \AgeK$. If $|A| = 2$ then whenever $\cA^* \cong \cA$ with $A^* \subseteq \Powerset_{n+1}(\Set)$ we have $A^*$ is a sunflower. Therefore $\cA$ witnesses that $\AgeK$ has the $n+1$-sunflower property at $\cA$ and $|A| \leq \gamma_{\cA}(1) \leq \gamma_{\cA}(n+1)$.  

Now suppose $|A| > 2$ and let $(a_i)_{i \in [\ell+1]}$ be an enumeration of $\cA$. For $i \in [\ell]$ let $p_i = \qftp(a_{i+1}/ (a_j)_{j \in [i]})$. Let $\cB^-$ witness that $\AgeK$ has the $n$-sunflower property at $\cA$ and $|B^-| \leq \gamma_{\cA}(n)$. 

Let $\cB \in \AgeK$ be a structure witnessing the indivisibility of $\cB^-$ in $\AgeK$ of order $(n+1)\cdot |A|$. Further suppose $|B| \leq \beta(|B^-|, (n+1)\cdot |A|) \leq \beta(\gamma_{\cA}(n), (n+1)\cdot |A|)$.

We now define, for each $i \in [\ell+1]$, a structure $\cC_i$ along with a collection $V_i$ of ``valid'' sequences of length $i$ with the property that if $\cc \in V_i$ then $\qftp(\cc) = \qftp(\<a_j\>_{j \in [i]})$.  Further we will have $|V_i| \leq |B|^i$ and $|C_i| \leq |B| \cdot \sum_{j < i} |V_j|$. 

Let $\cC_1 = \cA\rest[\{a_0\}]$ and let $V_1 = \{\<a_0\>\}$, i.e. $\<a_0\>$ is the unique valid sequence of length $1$.

Suppose we have defined $\cC_k$ and $V_k$ for $k \in [\ell]$. For every $\cc \in V_k$ we find a structure $\cB_\cc$ such that 
\begin{itemize}
\item $B_\cc \cap C_k = \cc$, 

\item $\cB_\cc \cong \cB$, 

\item for all $b \in \cB_\cc$ we have $\qftp(\cc\^b) = \qftp(\<a_i\>_{i \in [k]})$.  

\end{itemize}

By Lemma \ref{JEP and indivisible implies unqiue unary qf-types}, we have that $\AgeK$ has unique unary quantifier-free types and so we can always find such a $\cB_\cc$ by Lemma \ref{Use of 3-DAP}. Further note we can assume, without loss of generality, that for any two $\cc_0, \cc_1 \in V_k$ that $B_{\cc_0} \cap B_{\cc_1} \subseteq C_k$. 

Let $\cC_{k+1}$ be a disjoint amalgamation of all of these. Note such a disjoint amalgamation exists in $\AgeK$ by Lemma \ref{3-DAP implies amalgamation}. We therefore have $|C_{k+1}| = |C_k| + |B| \cdot |V_k| \leq |B| \cdot \sum_{i \in [k+1]} |V_i|$. Let $V_{k+1} = \{\cc \^b \st \cc \in V_k \text{ and } b \in \cB_{\cc}\}$. We then have $|V_{k+1}| \leq |B| \cdot |V_k| \leq |B|^{k+1}$. In particular this implies $|C_{\ell}|  \leq |B|\cdot \sum_{i \in [\ell+1]} |B|^i \leq \beta(\gamma_{\cA}(n), (n+1)\cdot |A|)^{\leq |A|}$.  

To finish the induction it suffices to prove the following claim. 
\begin{claim}
$\cC_\ell$ witnesses that $\AgeK$ has the $n+1$-sunflower property at $\cA$.
\end{claim}
\begin{proof}
Suppose $\cC^\circ_\ell$ is such that $\alpha\: \cC^\circ_\ell \cong \cC_\ell$ and $C^\circ_\ell \subseteq \Powerset_{n+1}(X)$ for some set $X$. Let $V_\ell^\circ = \{\alpha(\cc) \st \cc \in V_\ell\}$. We now use $V_\ell^\circ$ to attempt to construct a sunflower isomorphic to $\cA$. 

For $x \in X$, let $E_x = \{y \in C^\circ_\ell \st x \in y\}$. Call a sequence $\<e_i\>_{i \in [k]}$ \defn{separated} if for all $x \in X$ $|\{i \st e_i \in E_x\}| \leq 1$. We now have two cases. \nl\nl
\ul{Case $1$:} There is a $\<c_i\>_{i \in [\ell+1]} \in V_{\ell}$ which is separated. \nl
In this case we have the map $a_i \mapsto c_i$ is an isomorphism between $\cA$ and $\cC^\circ_\ell \rest[\{c_i\}_{i \in [\ell+1]}]$. Further, because $\<c_i\>_{i \in [\ell+1]}$ is separated, $c_i \cap c_j = \emptyset$ for $i < j \in [\ell+1]$. Therefore $\{c_i\}_{i \in [\ell+1]}$ is a sunflower. \nl\nl
\ul{Case $2$:} For all $\cc \in V_{\ell}$, $\cc$ is not separated. \nl
Let $P = \{\cc \in \bigcup_{i \in [\ell+1]}V_{i} \st \cc\text{ is separated}\}$. Let $\cc \leq \dd$ if $\cc$ is an initial sequence of $\dd$. $(P, \leq)$ is then a finite partial ordering and so must have a maximal element $\dd$. But by assumption, as $\dd$ is separated, $\dd \not \in V_{\ell}$. 

Let $\dd = \<d_i\>_{i \in [k]}$ where $k \in [\ell]$. Let $X_0 = \bigcup_{i \in [k]} d_i$, let $r = |X_0|$ and let $\gamma\:X_0 \to [r]$ be a bijection. Note $r \leq (n+1) \cdot |A|$. 

As $\dd$ is maximal among $\bigcup_{i \in [\ell+1]} V_i$ which are separated we have for all $b \in\cB_\dd$ that $\dd\^b$ is not separated. Therefore for all $b \in B_\dd$ we have $b \cap X_0 \neq \emptyset$. Let $\zeta\:B_\dd \to [r]$ be any map such that $(\forall b \in B_\dd)\, \gamma^{-1}(\zeta(b)) \in b$. 

But because $\cB$ is a witness to the indivisible of $\cB^-$ in $\AgeK$ of order $(n+1)\cdot |A|$ there must be an $s \in [r]$ and a subset $B_\dd^- \subseteq B_\dd \cap \zeta^{-1}(s)$ with $B_\dd^- \cong \cB^-$. Let $z = \gamma^{-1}(s)$. Therefore, for all $b \in B_\dd^-$ we have $z \in b$. 

Let $G_\dd^- = \{y \setminus \{z\} \st y \in B_\dd^-\}$. We then have the map $\nu_z\:G_\dd^- \to B_{\dd}^-$ where $\nu_z(y) = y \cup \{z\}$ is an isomorphism. Let $\cG_\dd^-$ be the $\Lang$-structure with underlying set $G_\dd^-$ where $\nu_z$ is an isomorphism of $\Lang$-structures. 

Because $\cB^-$ witnesses that $\AgeK$ has the $n$-sunflower property at $\cA$ and $\cG_\dd^- \cong \cB^-$ there must be an $\cA^* \subseteq \cG_\dd^-$ where $\cA^* \cong \cA$ and $A^*$ is a sunflower. 

But then if we let $\cA^\circ$ be $\nu_z[\cA^*]$, we have $\cA^\circ \subseteq \cC_\ell^\circ$ and $\cA^\circ \cong \cA$. But we also have for all $a, b, c \in \cA^*$, $\nu_z(a) \cap \nu_z(b) = (a \cap b) \cup \{z\} = (a \cap c) \cup \{z\} = \nu_z(a) \cap \nu_z(c)$. Therefore $A^\circ$ is also a sunflower. 

Therefore $\cC_\ell$ witnesses that $\AgeK$ has the $n+1$-sunflower property at $\cA$. 
\end{proof} 
\end{proof}

\section{Indivisibility}
\label{Indivisibility}

In this section, we discuss the relationship between indivisibility and being sunflowerable or being an age with a sunflower property. This will allow us to use our results to give new classes of structures which are indivisible. The following results follow from Proposition \ref{2-sunfloweable implies indivisible} and the corresponding facts about sunflowerable structures.

\begin{lemma}
Suppose 
\begin{itemize}

\item $\cM$ is an $\Lang$-structure with $|M|$ a regular cardinal, 

\item $\cM$ is an $|M|$-saturated structure with quantifier-elimination, 

\item $\cM$ has $\DAP(3)$.   
\end{itemize}
Then $\cM$ is indivisible. 
\end{lemma}
\begin{proof}
This follows immediately from Corollary \ref{Corollary showing when structure is sunflowerable from 3-DAP} and Proposition \ref{2-sunfloweable implies indivisible}. 
\end{proof}

Now we show that if an age has the sunflower property it is indivisible.

We now show that being indivisible is a necessary condition for having the $2$-sunflower property. 
\begin{proposition}
\label{Structure 2-sunflower property implies indivisibility}
If $\AgeK$ has the $2$-sunflower property, then $\AgeK$ is indivisible at $\cA$ whenever $\cA \in \AgeK$ with $|A| > 2$. 
\end{proposition}
\begin{proof}
Suppose $\cA \in \AgeK$ and let $\cB$ witness that $\AgeK$ has the $2$-sunflower property at $\cA$. Note we can assume without loss of generality that $\bigcup B \cap [2] = \emptyset$. Suppose $\alpha\:B \to [2]$. Let $\cB'\in \AgeK$ be the structure where $B' = \{\{b, \alpha(b)\}\st b \in B\}$ and where the projection map $\pi_0\:B' \to B$ is an isomorphism. There must then be a subset $A' \subseteq B'$ with $\cB\rest[A'] \cong \cA$ and $A'$ a sunflower. 

But as $|A| > 2$, all sunflowers in $B'$ are either subsets of $\{\{b, 0\}\st \alpha(b) = 0\}$ or subsets of $\{\{b, 1\} \st \alpha(b) = 1\}$. Therefore we must have $|\alpha``[A']| = 1$.

The result then follows from the following claim. 
\begin{claim}
Suppose for all $\cA \in \AgeK$ there is a $\cB_2 \in \AgeK$ which witnesses the indivisibility of $\cA$ in $\AgeK$ of order $2$. Then $\AgeK$ is indivisible. 
\end{claim}
\begin{proof}
We define $\cB_n$, for $n > 2$, to be such that $\cB_{n+1}$ witness the indivisibility of $\cB_n$ in $\AgeK$ of order $2$. We now prove by induction on $n$ that $\cB_n$ witnesses the indivisibility of $\cA$ in $\AgeK$ of order $n$. 

First note that if $n = 2$ then this follows from our definition of $\cB_2$. Now suppose $\cB_n$ witnesses the indivisibility of $\cA$ in $\AgeK$ of order $n$. 

Let $\alpha\:B_{n+1} \to [n+1]$.
Let $\gamma\:[n+1] \to [2]$ be such that $\gamma(i) =0$ for $i \in [n]$ and $\gamma(n) = 1$. By our definition of $\cB_{n+1}$, there must be a $B_n^* \subseteq B_{n+1}$ such that $\cB_n \cong \cB_{n+1}\rest[B_n^*]$ with $|(\gamma \circ \alpha)``[B_n^*]| = 1$. 

Suppose $B_n^* \subseteq (\gamma \circ \alpha)^{-1}(1)$ then $B_n^* \subseteq \alpha^{-1}(\{n\})$. But by construction there must be an $A^* \subseteq B_n^*$ with $\cB_n^*\rest[A^*] \cong \cA$. Therefore $|\alpha``[A^*]| = 1$ and we are done. 
Now suppose $B_n^* \subseteq (\gamma \circ \alpha)^{-1}(0)$. Then $\alpha\rest[B_n^*]\:B_n^* \to [n]$. Therefore by our inductive hypothesis there must be an $A^* \subseteq B_n^*$ such that $\cA \cong \cB_n^*\rest[A^*]$ and $|\alpha``[A^*]| = 1$. 
\end{proof}
\end{proof}
Note in Proposition \ref{Structure 2-sunflower property implies indivisibility} we cannot remove the assumption that $|A| > 2$ as every subset of $\Powerset_n(\Set)$ of size $2$ is a sunflower. 

As a consequence of Proposition \ref{Structure 2-sunflower property implies indivisibility} we also get the following which was proved in \cite{MR4723481} Proposition 2.23. Rehana Patel, using different methods, has independently shown the following result. 
\begin{lemma}
Suppose $\AgeK$ is an age with only countably many isomorphism classes, which has unique unary quantifier-free types, and which has \HP\ and \DAP(3). Then $\AgeK$ is indivisible. 
\end{lemma}
\begin{proof}
This follows from Theorem \ref{3-DAP implies sunflower property} and Proposition \ref{Structure 2-sunflower property implies indivisibility}
\end{proof}

\section{Conjectures}
\label{Conjectures}

We end with a couple of conjectures. First recall by Lemma \ref{n+2-sunflowrable implies n+1-sunflowerable} that the properties of being $n$-sunflowerable form a hierarchy. However, for all examples in this paper, this hierarchy collapses. This suggests the following conjecture.

\begin{conjecture}
If $\cM$ is a $2$-sunflowerable structure, then $\cM$ is sunflowerable. 
\end{conjecture}

Theorem \ref{Characterizing sunflowrable countable linear orderings} characterizes those countable linear orderings which are sunflowerable. However, when we move to uncountable linear orderings we do not quite have such a characterization as there are linear orderings which are $\Rationals_\kappa$-scattered but not $\kappa$-scattered. We therefore have the following conjecture

\begin{conjecture}
If $\kappa$ is regular then all $\kappa$-dense linear orderings are sunflowerable. 
\end{conjecture}

For $n \geq 3$ recall $\cH_n$ is the countable generic $K_n$-free graph. Then $\cH_n$ is ultrahomogeneous and indivisible (see \cite{MR1047784}). This then suggests the following conjecture.

\begin{conjecture}
For all $n \geq 3$, $\cH_n$ is sunflowerable. 
\end{conjecture}

\bibliographystyle{amsnomr}
\bibliography{bibliography}

\end{document}